\documentclass[12pt,reqno,twoside]{amsart}

\usepackage{amsfonts,amsmath,amssymb,amstext,amsbsy,amsopn,amsthm,paralist}

\usepackage{graphicx}
\usepackage{color}
\usepackage{mathtools}
\usepackage{dsfont}
\usepackage{etoolbox}

\usepackage{latexsym}
\usepackage{a4wide}

\makeatletter
\def\eqnarray{\stepcounter{equation}\let\@currentlabel=\theequation
\global\@eqnswtrue
\tabskip\@centering\let\\=\@eqncr
$$\halign to \displaywidth\bgroup\hfil\global\@eqcnt\z@
  $\displaystyle\tabskip\z@{##}$&\global\@eqcnt\@ne
  \hfil$\displaystyle{{}##{}}$\hfil
  &\global\@eqcnt\tw@ $\displaystyle{##}$\hfil
  \tabskip\@centering&\llap{##}\tabskip\z@\cr}

\def\endeqnarray{\@@eqncr\egroup
      \global\advance\c@equation\m@ne$$\global\@ignoretrue}

\def\@yeqncr{\@ifnextchar [{\@xeqncr}{\@xeqncr[5pt]}}
\makeatother

\newtheorem{theorem}{Theorem}[section]
\newtheorem{lemma}[theorem]{Lemma}
\newtheorem{corollary}[theorem]{Corollary}
\newtheorem{proposition}[theorem]{Proposition}

\theoremstyle{definition}
\newtheorem{definition}[theorem]{Definition}
\newtheorem{assu}[theorem]{Assumption}
\newtheorem*{assumptionD}{Assumption~D}
\newtheorem*{assumptionN}{Assumption~N}
\newtheorem*{assumptionO}{Assumption~$\Omega^c$}
\newtheorem*{assumptionP}{Assumption~P}
\newtheorem{notation}[theorem]{Notation}

\newtheorem{rem}[theorem]{Remark}

\newcommand{\field}[1]{\mathbb{#1}}

\newcommand{\R}{\field{R}}
\newcommand{\W}{\mathrm{W}}
\renewcommand{\L}{\mathrm{L}}
\newcommand{\C}{\mathrm{C}}
\newcommand{\IN}{\field{N}}
\newcommand{\IR}{\field{R}}
\newcommand{\IC}{\field{C}}

\newcommand{\eps}{\varepsilon}
\newcommand{\ii}{\mathrm{i}}
\newcommand{\e}{\mathrm{e}}
\renewcommand{\Re}{\mathrm{Re}}
\renewcommand{\Im}{\mathrm{Im}}
\newcommand{\cJ}{\mathcal{J}}
\renewcommand{\d}{\mathrm{d}}
\newcommand{\Sec}{\mathrm{S}}
\newcommand{\Lop}{\mathcal{L}}
\newcommand{\diam}{\mathrm{diam}}
\newcommand{\loc}{\mathrm{loc}}

\DeclareMathOperator{\dom}{dom}    %domain of an operator
\DeclareMathOperator{\dist}{dist}  %distance
\DeclareMathOperator{\supp}{supp}  %support of a function
\DeclareMathOperator{\dive}{div}   %divergence

\DeclareMathOperator*{\essinf}{essinf}

\newcommand{\ta}{\mathfrak{t}}           % \ta = \tn + \tb 
\renewcommand{\O}{\Omega}

\newcounter{teller}

\newenvironment{tabel}{\begin{list}%
{\rm  (\alph{teller})\hfill}{\usecounter{teller} \leftmargin=1.1cm
\labelwidth=1.1cm \labelsep=0cm \parsep=0cm}
                      }{\end{list}}

\newcommand{\one}{\mathds{1}}
\newcommand{\sgn}{\mathop{\rm sgn}}

\begin{document}

\title[$\L^p$-theory for complex coefficients]{On the $\L^p$-theory for second-order elliptic operators in divergence form with complex coefficients}

\author{\sc A.F.M. ter Elst}
\address{Department of Mathematics, University of Auckland, Private Bag 92019,
Auckland 1142, New Zealand}
\email{t.terelst@auckland.ac.nz}
\thanks{Part of this work is supported by the Marsden Fund Council from Government
funding, administered by the Royal Society of New Zealand.}

\author{\sc R. Haller-Dintelmann}
\address{Technische Universit\"at Darmstadt, Fachbereich Mathematik,
Schlossgartenstr.\@ 7, 64289 Darmstadt, Germany}
\email{haller@mathematik.tu-darmstadt.de}

\author{\sc J. Rehberg}
\address{Weierstrass Institute for Applied Analysis and Stochastics,
 Mohrenstr.\@ 39, 10117 Berlin, Germany}
\email{rehberg@wias-berlin.de}

\author{\sc P. Tolksdorf}
\address{Universit\'e Paris-Est Cr\'eteil, LAMA UMR CNRS 8050, France}
\email{patrick.tolksdorf@u-pec.fr}
\thanks{The fourth author was partially supported by the project ANR INFAMIE
(ANR-15-CE40-0011)}


\begin{abstract}
Given a complex, elliptic coefficient function we investigate for which values
of $p$ the corresponding second-order divergence form operator, complemented
with Dirichlet, Neumann or mixed boundary conditions, generates a strongly
continuous semigroup on $\L^p(\Omega)$. Additional properties like analyticity
of the semigroup, $\mathrm{H}^\infty$-calculus and maximal regularity are
also discussed. Finally we prove a perturbation result for real coefficients
that gives the whole range of $p$'s for small imaginary parts of the
coefficients. Our results are based on the recent notion of $p$-ellipticity,
reverse H\"older inequalities and Gaussian estimates for the real coefficients.
\end{abstract}

\begingroup
% Hack to get rid of the amsart all-caps style in the title.
% http://tex.stackexchange.com/questions/2820/disable-toggle-smallcaps
\makeatletter
\patchcmd{\@settitle}{\uppercasenonmath\@title}{\large}{}{}
\patchcmd{\@setauthors}{\MakeUppercase}{}{}{}
\makeatother
\maketitle
\endgroup

\section{Introduction}
One of the central items when considering elliptic operators is their 'parabolic behaviour', such as the
generator property of (analytic) semigroups (\cite{pazy},~\cite{Grie}) or
maximal parabolic regularity (\cite{Denk},~\cite{kunst}). 
In case of second-order divergence operators and \emph{real measurable}
coefficients very satisfactory results are available even in case of
non-smooth domains and mixed boundary conditions. This is due to the fact that
one can prove upper Gaussian estimates for the semigroup on~$\L^2$, see \cite{tom1}. From this one can
deduce that the semigroup extrapolates to a consistent semigroup on $\L^p$ for all 
$p \in [1,\infty]$, see~\cite[Ch.~7]{Ouh05}. In addition, maximal parabolic
regularity on $\L^p$ for all $p \in (1,\infty)$ can be shown 
(\cite{HiebPruess} and~\cite{coulhon}) and even a bounded
$\mathrm{H}^\infty$-calculus is obtained \cite{duongrobin}. Moreover, it can be shown that these 
semigroups on $\L^p$ are all
contraction semigroups (\cite[Ch.~4]{Ouh05} or~\cite{hoemb}). This then
allows for another proof of a bounded $\mathrm{H}^\infty$-calculus via
\cite{cowling} and for maximal parabolic regularity via~\cite{lambert}.

Unfortunately nearly all of this breaks down when admitting complex
coefficients. The only thing that obviously remains true is the fact that the
$\L^2$ semigroup extrapolates consistently to strongly continuous semigroups on the spaces
$\L^p$ for all $p \in (2_*, 2^*)$, where $2^*$ is the first Sobolev exponent
and $2_* = (2^*)'$ is the dual exponent.
This is a consequence of the inclusion $\W^{1,2} \subset \L^p$ if $p \in
[2, 2^*)$.

Apart of this, compared to the case of real-valued measurable coefficient functions
several severe obstructions appear. We list four of the most striking ones.
First, even if there is a consistent semigroup on $\L^\infty$, this need not be
a contraction semigroup~\cite{ABBO}. Secondly, the
'parabolic maximum principle' does not hold, as was pointed out in~\cite{auscher2}. 
Furthermore, an ingenious example in~\cite{mazya} 
shows that a distributional solution for the elliptic equation
with right-hand side $0$ is not necessarily locally bounded. Finally, the
semigroups may even cease to exist on an $\L^p$ space with finite $p$, see~\cite{davies}.

The aim of this paper is to investigate the following two questions: First,
given a bounded open set $\Omega \subset \R^d$ with $d \geq 2$, and a strongly
elliptic coefficient matrix $\mu \in \L^\infty(\Omega ; \IC^{d \times d})$, for
which $p$ does the elliptic divergence form operator $- A := \nabla \cdot \mu
\nabla$ complemented with appropriate boundary conditions generate a strongly
continuous semigroup on $\L^p$? Secondly, if this is true, what additional
properties, such as analyticity of the semigroup, bounded holomorphic
functional calculus, and maximal parabolic regularity does the operator have?

It is well-known that there exists an $\eps_0 > 0$ such that 
$-A$ generates an analytic semigroup on $\L^p (\Omega)$ for all $p \in (1,\infty)$ with 
\begin{align}
\label{Eq: Sobolev condition}
 \Big\lvert \frac{1}{p} - \frac{1}{2} \Big\rvert < \frac{1}{d} + \eps_0,
\end{align}
see~\cite{Davies_Analyticity, Blunck_Kunstmann, Auscher_Riesz, Tolksdorf}. 
 In the case $d = 2$,
 this condition already covers the $\L^p$-spaces for all $p \in (1 , \infty)$. 
 In general, however, condition~\eqref{Eq: Sobolev condition}
 is sharp, i.e., for each $p \in [1 , \infty)$ that satisfies $\lvert 1 / p - 1 / 2 \rvert > 1 / d$ there exists
a strongly elliptic 
coefficient function $\mu \in \L^{\infty} (\Omega ; \IC^{d \times d})$ such that $- A$ does not generate an analytic semigroup on
 $\L^p (\Omega)$, see~\cite{hofmann}. This raises the issue of quantifying the largeness of $\eps_0$ for given coefficients. 

In their pioneering works~\cite{cialdea/mazya, cialdea/mazya/systems, cialdea/mazya/lame}, Cialdea and Maz'ya 
found a purely algebraic condition between $\mu$ and $p$ that implies that $- A$ is accretive on $\L^p (\Omega)$. Subsequently, this condition was elegantly reformulated by Carbonaro and
 Dragi\v{c}evi\'c~\cite{Carbonaro_Dragicevic} as follows.
A coefficient function $\mu \in \L^{\infty} (\Omega ; \IC^{d \times d})$ is 
called $p$-\emph{elliptic} if there exists a $\lambda_p > 0$ such that 
\[
\Re \langle \mu (x) \xi , \cJ_p (\xi) \rangle \geq \lambda_p \lvert \xi \rvert^2 
\]
for almost every $x \in \Omega$ and $\xi \in \IC^d$,
where
\begin{align*}
 \cJ_p (\xi) \coloneqq 2 \bigg( \frac{\Re(\xi)}{p^{\prime}} + \frac{\ii \Im (\xi)}{p} \bigg) \qquad (\xi \in \IC^d).
\end{align*}
It is shown in~\cite{Carbonaro_Dragicevic}, 
that given a strongly elliptic matrix $\mu \in \L^{\infty} (\Omega ; \IC^{d \times d})$ there exists a unique number $2 < p_0 (\mu) \leq \infty$ such that $\mu$ is $p$-elliptic if and only if $p \in (p_0 (\mu)^{\prime} , p_0 (\mu))$, where $p_0 (\mu)^{\prime}$ denotes the H\"older conjugate exponent to $p_0 (\mu)$. 

If $A$ is complemented with mixed Dirichlet/Neumann boundary conditions we show in 
Theorem~\ref{Thm: Analyticity} that under very general conditions on $\Omega$ the operator 
$- A$ generates a bounded analytic semigroup on $\L^p (\Omega)$ if
\begin{align}
\label{Eq: Analyticity interval}
 \frac{p_0 (\mu) d}{d (p_0 (\mu) - 1) + 2} < p < \frac{p_0 (\mu) d}{d - 2}.
\end{align}
This gives a lower bound on $\eps_0$, namely
\begin{align*}
 \eps_0 \geq \frac{(p_0 (\mu) - 2) (d - 2)}{2 p_0 (\mu) d}.
\end{align*}
The proof of Theorem~\ref{Thm: Analyticity} relies on the verification of weak reverse H\"older estimates for the resolvent operators $(\lambda + A)^{-1}$ which combined with Shen's $\L^p$-extrapolation theorem~\cite{Shen} extrapolates $\L^2$-resolvent estimates to $\L^p$. The proof of these weak reverse H\"older estimates relies on a Moser-type iteration scheme and it was the insight of Cialdea and Maz'ya~\cite{cialdea/mazya} that the $p$-ellipticity condition is just the right condition that allows to test the equation with a testfunction of the form $\lvert u \rvert^{p - 2} u$. Subsequently, a localized version of this testfunction was used by Dindo\v{s} and Pipher~\cite{Dindos_Pipher} to prove the validity of weak reverse H\"older estimates of solutions $u$ that satisfy $- \nabla \cdot \mu \nabla u = 0$ in interior balls. In Theorem~\ref{Thm: Reverse Holder}, we give an adapted argument of how to establish weak reverse H\"older estimates for balls centred at the boundary and also for the resolvent equation. 

Notice that Theorem~\ref{Thm: Analyticity} was independently proven by Egert
in~\cite{Moritz_neu} by a different approach. Another version of Theorem~\ref{Thm: Analyticity}
was proved in~\cite{ELSV} with a probabilistic viewpoint instead of explicit boundary conditions. 

In Theorem~\ref{tnumrange307} we present a perturbation result for real-valued coefficients. 
We show that, given a real-valued elliptic matrix $\mu$, there exists an $\eps > 0$ such that 
for every  $\nu \in \L^{\infty} (\Omega ; \IC^{d \times d})$ 
with $\| \nu \|_{\L^{\infty} (\Omega ; \Lop(\IC^d))} < \eps$ the operator $- A$ associated to the matrix $\mu + \nu$ 
still generates an analytic semigroup on $\L^p (\Omega)$ for all $p \in [1 , \infty)$. 
As a corollary one obtains the 
existence of a `threshold' $p_c > 2$ depending only on the geometry and the ellipticity constants 
of a complex-valued $\mu$ such that, whenever $p_0 (\mu) > p_c$, the operator $- A$
generates an analytic semigroup on $\L^p (\Omega)$ for all $p \in (1 , \infty)$.

In our approach one of the the central
instruments are De Giorgi estimates. In view of our results, it seems not
accidental that in~\cite{mazya} the lack of $\L^\infty$-bounds for the solution
of the elliptic equation is brought in connection with the breakdown of the
classical De Giorgi arguments. 

The bounded analyticity of the semigroup $(\e^{- t A})_{t \geq 0}$ combined with results in~\cite{egert} have twofold consequences. One is that $A$ viewed as an operator on $\L^p (\Omega)$ with $p$ subject to~\eqref{Eq: Analyticity interval} admits the property of maximal $\L^q$-regularity for all $1 < q < \infty$. The second consequence is that the operator $A$ viewed as an operator on the negative scale $\W^{-1 , p}_D (\Omega)$ admits maximal $\L^q$-regularity for all $1 < q < \infty$ and $p \geq 2$ subject to~\eqref{Eq: Analyticity interval}. This is described in Corollary~\ref{Cor: Maximal regularity on Lp} and Section~\ref{Sec: Regularity for the induced operators on the W^{-1,q}_D scale}, respectively. 

Finally, in Theorem~\ref{Thm: Gradient estimates for resolvent}, we give optimal bounds 
for the operator 
$\nabla (\lambda + A)^{-1}$ in $\L^p (\Omega)$ and the operator
$\nabla (\lambda + A)^{-1} \dive$ 
in $\L^p (\Omega, \IC^d)$, where $2 \leq p < 2 + \eps$ for some $\eps > 0$. 
These estimates imply gradient estimates for the semigroup operators. Notice that this is an analogue of the higher integrability statement of Meyers in~\cite{Meyers}. We remark that the proof also works in two dimensions and for elliptic systems. It is already known (in the case $\lambda = 0$) that the bound $p < 2 + \eps$ is sharp, see, e.g., the striking counterexample in the plane in~\cite{Astala_Faraco_Szekelyhidi}. 

The article is organised as follows. 
In Section~\ref{Sec: Notation and preliminary results} we introduce the geometric setup,
discuss the $p$-ellipticity condition and present some preparatory lemmas. 
In Section~\ref{Sec: Formulation of the main results} we formulate our main results and 
in Section~\ref{Sec: Reverse Holder estimates for the resolvent equations} we prove
all weak reverse H\"older estimates that are needed to extrapolate 
estimates from $\L^2 (\Omega)$ to $\L^p (\Omega)$. 
In Section~\ref{Sec: From weak reverse Holder estimates to operator bounds} we prove
Theorems~\ref{Thm: Analyticity} and~\ref{Thm: Gradient estimates for resolvent}, and 
in Section~\ref{Sec: Perturbation of real-valued matrices} we prove the perturbation result 
of real-valued matrices, i.e., Theorem~\ref{tnumrange307}. 
Finally, in Section~\ref{Sec: Regularity for the induced operators on the W^{-1,q}_D scale} we discuss 
the transference of the maximal regularity to the $\W^{-1,q}_D$-scale.

\section{Notation and preliminary results}
\label{Sec: Notation and preliminary results}

Throughout this paper the space dimension $d \geq 3$ is fixed. All Sobolev and Lebesgue spaces are 
considered as Banach spaces over the complex field. 
If $A \subset \IR^d$ is bounded and Lebesgue measurable with its Lebesgue measure 
$\lvert A \rvert > 0$ and if $f \in \L^1_{\loc} (\IR^d)$, then denote by $f_A \coloneqq \lvert A \rvert^{-1} \int_A f \; \d x$ the mean value of $f$ over $A$. The characteristic function of the set $A$ is denoted by $\chi_A$. 

In the following, we consider a bounded and open set $\Omega \subset \IR^d$ along with a closed subset $D \subset \partial \Omega$ of its boundary. The subset $D$ corresponds to the boundary part where Dirichlet boundary conditions are prescribed and will be called the Dirichlet boundary. The complementary part $N \coloneqq \partial \Omega \setminus D$ is called the Neumann boundary. We denote by
\begin{align*}
 \C_D^{\infty} (\Omega) \coloneqq \{ \varphi|_{\Omega} : \varphi \in \C_c^{\infty} (\IR^d) \text{ and } \supp(\varphi) \cap D = \emptyset \}
\end{align*}
the space of all smooth functions that vanish on $D$. 
For all $1 \leq p < \infty$ we denote by $p^{\prime}$ the conjugate exponent to $p$ and we denote by
\begin{align*}
 \W^{1 , p}_D (\Omega) \coloneqq \overline{\C_D^{\infty} (\Omega)}^{\W^{1 , p} (\Omega)}
\end{align*}
the first-order Sobolev space of functions that vanish on $D$. 
By the very definition, it is clear that $\W^{1 , p}_D (\Omega)$ is invariant under 
multiplication by smooth and compactly supported functions. 
For the record we state the following two lemmas.

\begin{lemma}
\label{Lem: Multiplication by smooth functions}
Let $1 \leq p < \infty$ and $u \in \W^{1 , p}_D (\Omega)$. If $\eta \in \C_c^{\infty} (\IR^d)$, then $\eta u \in \W^{1 , p}_D (\Omega)$.
\end{lemma}

\begin{lemma}
\label{Lem: Modulus lemma}
Let $1 < p < \infty$ and $u \in \W^{1 , p}_D (\Omega)$.
Then $\lvert u \rvert \in \W^{1 , p}_D (\Omega)$.
\end{lemma}

\begin{proof}
Let $n \in \IN$.
There exists a $\varphi_n \in \C_c^\infty(\IR^d)$ such that 
$D \cap \supp \varphi_n = \emptyset$ and, moreover,
$\|\varphi_n|_\Omega - u\|_{\W^{1,p}(\Omega)} \leq \frac{1}{n}$.
Let $\varepsilon_n > 0$ and define $\psi_n = \sqrt{|\varphi_n|^2 + \varepsilon_n^2} - \varepsilon_n$.
Clearly $\psi_n \in \C^\infty(\IR^d)$ and $\supp \psi_n = \supp \varphi_n$.
So $\psi_n \in \C_c^\infty(\IR^d)$.
Write $w_n = \psi_n|_\Omega$.
Then $w_n \in \W^{1,p}_D(\Omega)$.
Choose $\varepsilon_n > 0$ such that $\|\psi_n - |\varphi_n| \|_{\L^p(\IR^d)} \leq \frac{1}{n}$.
Now $\nabla \psi_n = \frac{\Re(\overline{\varphi_n} \nabla \varphi_n)}{\sqrt{|\varphi_n|^2 + \varepsilon_n^2}}$,
so $|\nabla \psi_n| \leq |\nabla \varphi_n|$ and 
$\|w_n\|_{\W^{1,p}_D(\Omega)} \leq \|u\|_{\W^{1,p}(\Omega)} + \frac{2}{n}$.
Therefore the sequence $(w_n)_{n \in \IN}$ is bounded in $\W^{1,p}_D(\Omega)$ and it 
has a weakly convergent subsequence in $\W^{1,p}_D(\Omega)$.
The weak limit is an element of $\W^{1,p}_D(\Omega)$.
But $\lim_{n \to \infty} w_n = |u|$ in $\L^p(\Omega)$. 
Hence the weak limit is equal to $|u|$.
Consequently $|u| \in \W^{1,p}_D(\Omega)$.
\end{proof}

In the following, we define as usual $\W^{-1 , p}_D (\Omega)$ to be the antidual space of $\W^{1 , p^{\prime}}_D (\Omega)$ whenever $1 < p < \infty$.

\subsection{The geometric setup and related inequalities}

The set $\Omega$ under consideration is supposed to satisfy the following condition at the 
closure of the Neumann boundary.

\begin{assumptionN}
There exists a constant $M \geq 1$ such that for every $x \in \overline{N}$ there exist 
an open neighbourhood $U_x \subset \IR^d$ of $x$ and a 
bi-Lipschitz homeomorphism $\Phi_x$ from an open neighbourhood of $\overline{U_x}$ onto an open subset of $\IR^d$ such that $\Phi_x (x) = 0$,
\begin{align*}
 \Phi_x (U_x) &= (-1 , 1)^d, \\
 \Phi_x (U_x \cap \Omega) &= (-1 , 1)^{d - 1} \times (0 , 1), \\
 \Phi_x (U_x \cap \partial \Omega) &= (-1 , 1)^{d - 1} \times \{ 0 \},
\end{align*}
and such that the Lipschitz constants of $\Phi_x$ and $\Phi_x^{-1}$ are both less than $M$.
\end{assumptionN}

We emphasise that the constant~$M$ is independent of the point~$x$.

In the following, we will frequently intersect $\Omega$ with a ball $B(x , r)$ so that we introduce the short-hand notation
\begin{align*}
 \Omega (x , r) \coloneqq \Omega \cap B(x , r).
\end{align*}

\begin{rem}
\label{Rem: Reflection at Lipschitz boundary}
Let $x \in \overline{N}$. 
Then Assumption~N allows to construct a local extension operator $E$ in a neighbourhood 
of $x$ by reflection at the Lipschitz boundary, 
see~\cite[Prop.~2.3 and Rem.~2.2]{Tolksdorf}. 
More precisely, given $0 < r \leq 1 / 4$ there exists a linear operator $E$ that 
maps measurable functions on $\Omega$ to measurable functions on 
$\Omega \cup B(x , r / (M \sqrt{d}))$ which satisfies for all $1 \leq p < \infty$ the estimates
\begin{align*}
\begin{aligned}
 \| E f \|_{\L^p (B(x , r / (M \sqrt{d}))} &\leq C \| f \|_{\L^p (\O (x , M r))} &&(f \in \L^p (\Omega)) \\
 \| \nabla E f \|_{\L^p (B(x , r / (M \sqrt{d}) ; \IC^d)} &\leq C \| \nabla f \|_{\L^p (\O (x , M r) ; \IC^d)} &&(f \in \W^{1 , p}_D (\Omega)).
\end{aligned}
\end{align*}
Here the constant $C > 0$ depends only on $d$, $p$ and $M$.
\end{rem}

In some situations it is desirable to have Sobolev's embedding theorem for a function 
$u \in \W^{1 , p}_D (\Omega)$ on sets of the 
form $\Omega (x_0 , r)$ available, where
$x_0 \in \partial \Omega$ and $r > 0$ is small enough. 
The set $\Omega (x_0 , r)$, however, 
might be very irregular as it might contain cuspidal boundary points at boundary parts
where $u$ does not vanish. A way out is guaranteed if one allows the domain of integration in the integral on the right-hand side of Sobolev's inequality to be slightly enlarged. Indeed, this allows to introduce a suitable superset of $\Omega (x_0 , r)$ which is regular enough to employ Sobolev's embedding theorem there. A quantitative version of this argument is presented in the following lemma whose proof can be found in~\cite[Lem.~5.4]{Tolksdorf}.

\begin{lemma}
\label{Lem: Local Sobolev embedding}
Let $\Omega \subset \IR^d$ be open and bounded and $D \subset \partial \Omega$ be subject to Assumption~N. Let $p \in [1 , \infty)$ and $u \in \W^{1 , p}_D (\Omega)$. Let $x_0 \in \overline{\Omega}$ and $0 < r \leq 1 / (4 M \sqrt{d})$ be such that either $B(x_0 , r) \subset \Omega$ or $x_0 \in \partial \Omega$. If $q \in [1 , \infty)$ is such that $0 \leq 1 / p - 1 / q \leq 1 / d$, then there exists a 
constant $C_{\rm Sob} > 0$, depending only on $p$, $q$, $d$ and $M$, such that
\begin{align*}
 \bigg( \frac{1}{r^d} \int_{\Omega (x_0 , r)} \lvert u \rvert^q \; \d x \bigg)^{1 / q} 
\leq C_{\rm Sob} \bigg\{ r \bigg( \frac{1}{r^d} \int_{\Omega (x_0 , \alpha r)} \lvert \nabla u \rvert^p \; \d x \bigg)^{1 / p} + \bigg( \frac{1}{r^d} \int_{\Omega (x_0 , r)} \lvert u \rvert^p \; \d x \bigg)^{1 / p} \bigg\} ,
\end{align*}
where $\alpha \coloneqq M^2 \sqrt{d}$.
\end{lemma}

\begin{rem}
We will use Sobolev's embedding only in the cases $p = 2$ and $q = 2 d / (d - 2)$, or $p = 2 d / (d + 2)$ and $q = 2$. In these cases $C_{\text{Sob}}$ depends only on $d$ and $M$.
\end{rem}

To obtain higher integrability properties of the gradient of the solution to elliptic equations a further regularity property of $\Omega$ will be required in the proof. This property is the so-called corkscrew or plumpness condition of $\Omega^c$.

\begin{assumptionO}
\label{Ass: Plumpness}
There exist $r_0 > 0$ and $\kappa \in (0 , 1)$ such that for all $x \in \partial \Omega$ and $0 < r < r_0$ there exists an $x^* \in \Omega^c$ such that $B(x^* , \kappa r) \subset \Omega^c \cap B(x , r)$.
\end{assumptionO}

\begin{rem}
Notice that if $\Omega$ satisfies Assumption~N, then the plumpness condition is automatically satisfied in a neighbourhood of $\overline{N}$ due to the existence of the bi-Lipschitz coordinate charts. 
Thus, in this case, Assumption~$\Omega^c$ only introduces a condition to the behaviour of $\Omega$ at the Dirichlet boundary $D$.
\end{rem}

Another condition that is needed for the higher integrability property of the gradient is a plumpness condition for the Dirichlet boundary. This prevents the interface that separates $D$ and $N$ to have cusps that reach into the Neumann part.

\begin{assumptionD}
There exist $s_0 > 0$ and $\iota \in (0 , 1)$ such that for all $x \in D \cap \overline{N}$ and $0 < r < s_0$ there exists an $x^* \in D \cap B(x , r)$ such that $B(x^* , \iota r) \cap N = \emptyset$.
\end{assumptionD}

Under Assumptions~N,~$\Omega^c$ and~D one can prove the following Poincar\'e-type inequality close to the Dirichlet boundary.

\begin{lemma}
\label{Lem: Local Poincare}
Let $\Omega \subset \IR^d$ be a bounded, open set and $D \subset \partial \Omega$ be closed and 
subject to Assumptions~N,~$\Omega^c$ and~D. 
Let $p \in [1,\infty)$. 
Then there exists a constant $C > 0$ such that 
\begin{align*}
 \| u \|_{\L^p (\Omega (x , r))} \leq C r \| \nabla u \|_{\L^p(\Omega (x , \beta r) ; \IC^d)},
\end{align*}
for all $u \in \W^{1 , p}_D (\Omega)$, $x \in \partial \Omega$ and 
$0 < r < \min\{ s_0 / 2 , r_0 / (2 \iota) , r_0 , 1 / (8M \sqrt{d}) \}$ 
with $B(x , r) \cap D \neq \emptyset$,
where $\beta \coloneqq 4 M^2 \sqrt{d}$.
Here $C > 0$  depends only on $d$, $M$, $p$, $\kappa$ and $\iota$.
\end{lemma}

\begin{proof}
First of all, let $x \in \partial \Omega$ and $0 < r < r_0$ be such that $B(x , r) \cap N = \emptyset$. 
Let $u \in \W^{1 , p}_D (\Omega)$.
Let $\widetilde u$ be the extension by zero to $B(x , r)$ of 
$u|_{\Omega (x , r)}$.
Then $\widetilde u \in \W^{1 , p}_D (\Omega)$ by~\cite[Lem.~2.2(b)]{ERe2}.
Assumption~$\Omega^c$ implies the existence of a point $x^* \in \Omega^c$ such that $B(x^* , \kappa r) \subset \Omega^c \cap B(x , r)$. 
By \cite[Lem.~7.16]{Gilbarg_Trudinger} one deduces that 
\begin{align*}
 \lvert \widetilde{u} (y) - \widetilde{u}_{B(x^* , \kappa r)} \rvert \leq \frac{2^d}{d \lvert B(0 , 1) \rvert \kappa^d} \int_{B(x , r)} \lvert y - z \rvert^{-(d - 1)} \lvert \nabla \widetilde{u} (z) \rvert \; \d z
\end{align*}
for almost every $y \in B(x , r)$.
Notice that $\widetilde{u}$ is zero on $B(x^* , \kappa r)$ so that the mean value integral on the left-hand side is zero. Taking $\L^p$-norms with respect to the variable $y \in B(x , r)$  together with the 
boundedness of the Riesz potential \cite[Lem.~7.12]{Gilbarg_Trudinger} yields
\begin{align}
\label{Eq: First case inequality}
 \| u \|_{\L^p (\O(x , r))} = \| \widetilde{u} \|_{\L^p (B(x , r))} \leq \frac{2^d}{\kappa^d} r \| \nabla \widetilde{u} \|_{\L^p(B(x , r) ; \IC^d)} = \frac{2^d}{\kappa^d} r \| \nabla u \|_{\L^p(\O (x , r) ; \IC^d)}
\end{align}
as required. 

For all $x \in \partial \Omega$ and $0 < r < \min\{ s_0 / 2 , r_0 / (2 \iota) , r_0 , 1 / (8M) \}$ 
that satisfy $B(x , r) \cap N \neq \emptyset$ and $B(x , r) \cap D \neq \emptyset$ the
 situation is reduced to the previous case as follows. 
First of all, notice that there exists a $z \in \overline{N} \cap D \cap B(x , r)$ and that $B(x , r) \subset B(z , 2 r)$. Employ Assumption~D to obtain a point $z^* \in D \cap B(z , 2 r)$ such that $B(z^* , 2 \iota r) \cap N = \emptyset$. Moreover, let $E$ denote the local extension operator described in Remark~\ref{Rem: Reflection at Lipschitz boundary}. By employing the fact $E u = u$ on $\O (x , r)$ and the triangle and H\"older's inequality together with $B(x , r) \subset B(z , 2 r)$ one estimates
\begin{align}
\label{Eq: Triangle inequality for Poincare}
 \| u \|_{\L^p(\O(x , r))} \leq \| E u - (E u)_{\O (z^* , 2 \iota r)} \|_{\L^p(B(z , 2 r))} + \frac{\lvert B(x , r) \rvert^{1 / p}}{\lvert \O (z^* , 2 \iota r) \rvert^{1 / p}} \| u \|_{\L^p(\O (z^* , 2 \iota r))}.
\end{align}
Note that Assumption~N implies that
\begin{align}
\label{Eq: d-set near Neumann}
 \frac{\lvert B(x , r) \rvert^{1 / p}}{\lvert \O (z^* , 2 \iota r) \rvert^{1 / p}} \leq C
\end{align}
with a constant $C > 0$ that depends only on $\iota$, $d$, $M$ and $p$.
Then the second term on the right-hand side of~\eqref{Eq: Triangle inequality for Poincare} is estimated by~\eqref{Eq: First case inequality}.
It remains to control the term
\begin{align*}
\| E u - (E u)_{\O (z^* , 2 \iota r)} \|_{\L^p(B(z , 2 r))}.
\end{align*}
This is done as in the first part of the proof by virtue of~\cite[Lem.~7.12/7.16]{Gilbarg_Trudinger} and gives together with~\eqref{Eq: d-set near Neumann} a constant $C > 0$ that depends only on $d$, $p$, $M$ and $\iota$ such that
\begin{align*}
 \| E u - (E u)_{\O (z^* , 2 \iota r)} \|_{\L^p(B(z , 2 r))} \leq C r \| \nabla E u \|_{\L^p(B (z , 2 r) ; \IC^d)}.
\end{align*}
Employing Remark~\ref{Rem: Reflection at Lipschitz boundary} and $B(z , 2 M^2 \sqrt{d} r) \subset B(x , 4 M^2 \sqrt{d} r)$ concludes the proof.
\end{proof}

For later purposes, we introduce some more geometric concepts. To do so, denote by $\mathcal{H}^{d - 1}$ the $(d - 1)$-dimensional Hausdorff measure defined on $\IR^d$.

\begin{definition}
\begin{enumerate}
 \item An open or closed set $\Xi \subset \IR^d$ is called a $d$-set if there exists a 
$c > 0$ such that $\lvert \Xi \cap B(x , r) \rvert \geq c r^d$ for all $x \in \Xi$ and $0 < r < 1$. 
 \item A closed set $E \subset \IR^d$ is called a $(d - 1)$-set 
if there exist $C , c > 0$ such that 
$c r^{d - 1} \leq \mathcal{H}^{d - 1} (E \cap B(x , r)) \leq C r^{d - 1}$ for all $x \in E$ and $0 < r < 1$.
\end{enumerate}
\end{definition}

Moreover, we formulate another assumption. To do so, denote by $B^{\prime} (x^{\prime} , r)$ the open ball in $\IR^{d - 1}$ with radius $r > 0$ and centre $x^{\prime}$.

\begin{assumptionP}
Let $\Omega \subset \IR^d$ be open, $D \subset \partial \Omega$ be closed, and set $N \coloneqq \partial \Omega \setminus D$. If $\Omega$ is subject to Assumption~N and $x_0 \in \overline{N} \cap D$ let $\Phi_{x_0}$ denote the corresponding bi-Lipschitz homeomorphism with corresponding set $U_{x_0}$. Then there are $c_0 \in (0,1)$ and $c_1 > 0$ such that
\[
\mathcal{H}_{d-1} \{y^{\prime} \in B^{\prime}(x^{\prime} , r) :
\dist((y^{\prime},0),\Phi_{x_0}(U_{x_0} \cap N)) > c_0 r \} 
\geq c_1 r^{d-1}
\]
for all $r \in (0,1]$ and 
$x^{\prime} \in \IR^{d-1}$ with 
$(x^{\prime} , 0) \in  \Phi_{x_0}(U_{x_0} \cap D \cap \overline N)$.
\end{assumptionP}

\begin{rem}
Notice that Assumption~D implies Assumption~P.
\end{rem}

\begin{rem}
All stated geometric conditions are fulfilled in case that $\Omega$ is a bounded Lipschitz 
domain and in case that the interface separating $D$ and $N$ is Lipschitz as well. 
In particular the cases $D = \emptyset$ and $D = \partial \Omega$ are included, which will 
give Neumann and Dirichlet boundary conditions for the elliptic operator below.
\end{rem}

\subsection{The elliptic operator}
\label{Sec: The elliptic operator}

The operator under consideration is an elliptic operator $- \nabla \cdot \mu \nabla$ in divergence form. The matrix of coefficients is assumed to satisfy the following standard conditions.

\begin{assu}
\label{Ass: Ellipticity}
Let $\mu \in \L^{\infty} (\Omega ; \IC^{d \times d})$ be such that there exist numbers $c_{\bullet} , c^{\bullet} > 0$ with
\begin{align*}
 \| \mu \|_{\L^{\infty} (\Omega ; \Lop(\IC^d))} \leq c^{\bullet} \qquad \text{and} \qquad \Re \langle \mu(x) \xi , \xi \rangle \geq c_{\bullet} \lvert \xi \rvert^2 \quad (\xi \in \IC^d, \text{ a.e.\@ } x \in \Omega).
\end{align*}
\end{assu}

Let $\mu \in \L^{\infty} (\Omega ; \IC^{d \times d})$ be 
subject to Assumption~\ref{Ass: Ellipticity}.
Define the sesquilinear form
\begin{align}
\label{Eq: Sesquilinear form}
 \ta \colon \W^{1 , 2}_D (\Omega) \times \W^{1 , 2}_D (\Omega) \to \IC, 
\quad \ta[u , v] = \int_{\Omega} \langle \mu \nabla u , \nabla v \rangle \; \d x.
\end{align}
Let $A_2$ be the operator in $\L^2(\Omega)$ associated to $\ta$.
It is defined as follows.
Let $u,f \in \L^2 (\Omega)$. 
Then by definition $u \in \dom(A_2)$ and $A_2 u = f$ if and only if $u \in \W^{1 , 2}_D (\Omega)$ and 
$\ta[u , v] = \int_{\Omega} f \overline{v} \; \d x$ for all $v \in \W^{1 , 2}_D (\Omega)$.
It is classical that $A_2$ is a sectorial operator in $\L^2 (\Omega)$,
 i.e., there exists a $\theta \in (\pi / 2 , \pi)$ such that 
\begin{align*}
 \Sec_{\theta} \coloneqq \{ z \in \IC \setminus \{ 0 \} : \lvert \arg(z) \rvert < \theta \} \subset \rho(- A_2) 
\end{align*}
and there exists a constant $C > 0$ such that
\begin{align*}
 \| \lambda (\lambda + A_2)^{-1} f \|_{\L^2 (\Omega)} \leq C \| f \|_{\L^2 (\Omega)} 
\qquad (\lambda \in \Sec_{\theta}, \; f \in \L^2 (\Omega)).
\end{align*}
Here the constants $C$ and $\theta$ can be chosen to depend only on $c_{\bullet}$, $c^{\bullet}$ and $d$.

To investigate this operator on $\L^p (\Omega)$ one defines for all $p > 2$ the operator $A_p$ by
\begin{align*}
 \dom (A_p) &\coloneqq \{ u \in \dom(A_2) \cap \L^p (\Omega) : A_2 u \in \L^p(\Omega) \}, \\
 A_p u &\coloneqq A_2 u \quad (u \in \dom(A_p)),
\end{align*}
i.e., $A_p$ is the part of $A_2$ in $\L^p(\Omega)$. 

If $1 < p < 2$, then $A_p$ is defined as the closure of $A_2$ whenever $A_2$ is closable in $\L^p(\Omega)$. To decide whether $A_2$ is closable in $\L^p(\Omega)$ one can check whether the adjoint operator on $\L^{p^{\prime}} (\Omega)$ is densely defined
by using the following lemma~\cite[Lem.~2.8]{Tolksdorf}).

\begin{lemma}
\label{Lem: Duality}
Let $\Lambda \subset \IR^d$ be a bounded open set, $p \in (1 , 2)$
and let $B$ be a densely defined operator in $\L^2(\Lambda)$. 
Then $\dom(B)$ is dense in $\L^p(\Lambda)$ and $B$ is closable in $\L^p(\Lambda)$ 
if and only if the part of $B^*$ in $\L^{p^{\prime}}(\Lambda)$ is densely defined. 
In this case $(B_p)^* = (B^*)_{p^{\prime}}$, where $B_p$ denotes the closure of $B$ in $\L^p(\Lambda)$ and $(B^*)_{p^{\prime}}$ denotes the part of $B^*$ in $\L^{p^{\prime}} (\Lambda)$.
\end{lemma}

To obtain information about the numbers $p \in (1 , \infty)$ for which a reasonable $\L^p$-theory of the elliptic operator can be established, we introduce the notion of $p$-ellipticity. The origin of $p$-ellipticity is contained in the pioneering works of 
Cialdea and Maz'ya~\cite{cialdea/mazya, cialdea/mazya/systems, cialdea/mazya/lame, CiaM2}. In~\cite{cialdea/mazya}, Cialdea and Maz'ya show that the algebraic condition
\begin{align*}
 \frac{4}{p p^{\prime}} \langle \Re(\mu(x)) \alpha , \alpha \rangle + \langle \Re(\mu(x)) \beta , \beta \rangle + \frac{2}{p} \langle \Im(\mu(x)) \alpha , \beta \rangle - \frac{2}{p^{\prime}} \langle \alpha , \Im(\mu(x)) \beta \rangle \geq 0
\end{align*}
for almost every $x \in \Omega$ and all $\alpha , \beta \in \IR^d$ is sufficient for the operator 
$A_2$ to be an accretive operator in $\L^p (\Omega)$. 
Here $\Re(\mu(x))$ is the matrix obtained by taking the real part of each matrix element
of $\mu(x)$.
Similarly $\Im(\mu(x))$ is defined.
The term on the left-hand side can be written as
\begin{align}
\label{Eq: Lp dissipativity recasted into p-ellipticity}
\begin{aligned}
 \frac{4}{p p^{\prime}} \langle \Re(\mu(x)) \alpha , \alpha \rangle + \langle \Re(\mu(x)) \beta , \beta \rangle + \frac{2}{p} \langle \Im(\mu(x)) \alpha , \beta \rangle &- \frac{2}{p^{\prime}} \langle \alpha , \Im(\mu(x)) \beta \rangle \\
 &= p \Re \Big\langle \mu(x) ( \alpha^{\prime} + \ii \beta ) , \frac{\alpha^{\prime}}{p^{\prime}} + \frac{\ii \beta}{p} \Big\rangle,
\end{aligned}
\end{align}
where $\alpha^{\prime} = 2 \alpha / p$. 
The term on the right-hand side was investigated thoroughly by 
Carbonaro and Dragi\v{c}evi\'c in~\cite{Carbonaro_Dragicevic} and we next 
introduce their concepts.

\begin{definition}
Let $p \in (1 , \infty)$. 
Define $\cJ_p \colon \IC^d \to \IR^d$ by
\begin{align*}
 \cJ_p(\xi) = 2 \bigg(\frac{\alpha}{p^{\prime}} + \frac{\ii \beta}{p} \bigg),
\end{align*}
where $\xi = \alpha + \ii \beta$ with $\alpha , \beta \in \IR^d$.
Following~\cite{Carbonaro_Dragicevic} the matrix $\mu$ is called 
\textit{$p$-elliptic} if there exists a number $\lambda_p > 0$ such that
\begin{align}
\label{Eq: p-ellipticity inequality}
 \Re \langle \mu(x) \xi , \cJ_p (\xi) \rangle \geq \lambda_p \lvert \xi \rvert^2 
\qquad (\xi \in \IC^d, \text{ a.e.\@ } x \in \Omega).
\end{align}
\end{definition}

Obviously, a matrix $\mu$ subject to Assumption~\ref{Ass: Ellipticity} is always $2$-elliptic 
and if $\mu$ is in additional real-valued, then $\mu$ is $p$-elliptic for all $p \in (1 , \infty)$, see Lemma~\ref{Lem: Control of imaginary part} below. Next, define
\begin{align*}
 \Delta_p (\mu) \coloneqq \essinf_{x \in \Omega} \min_{\lvert \xi \rvert = 1} \Re \langle \mu(x) \xi , \cJ_p (\xi) \rangle
\end{align*}
to be the \textit{$p$-ellipticity constant} of $\mu$. 
Then $\mu$ is $p$-elliptic if and only if $\Delta_p (\mu) > 0$. 
Moreover, 
\[
 \Delta_p (\mu) = \Delta_{p^{\prime}} (\mu)
\]
by Proposition~5.8 in~\cite{Carbonaro_Dragicevic}.
As a consequence, $\mu$ is $p$-elliptic if and only if it is $p^{\prime}$-elliptic. 
Moreover, by~\cite[Cor.~5.17]{Carbonaro_Dragicevic} the conjugate coefficient function
$\mu^*$ is $p$-elliptic if and only if $\mu$ is $p$-elliptic.

Define 
\begin{align*}
 \delta(\mu) \coloneqq \essinf_{(x , \xi) \in \Omega \times \IC^d \setminus \{0\}} \Re \frac{\langle \mu(x) \xi , \xi \rangle}{\lvert \langle \mu(x) \xi , \overline{\xi} \rangle \rvert}.
\end{align*}
In~\cite[Prop.~5.15]{Carbonaro_Dragicevic} Carbonaro and Dragi\v{c}evi\'c proved that 
\begin{align*}
 \frac{\Delta_p (\mu)}{c^{\bullet}} \leq \delta(\mu) - \Big\lvert 1 - \frac{2}{p} \Big\rvert \leq \frac{\Delta_p (\mu) \delta(\mu)}{c_{\bullet}}.
\end{align*}
This relation separates $\mu$ and $p$ and it implies that $\mu$ is $p$-elliptic if and only if
\begin{align*}
 \Big\lvert 1 - \frac{2}{p} \Big\rvert < \delta(\mu).
\end{align*}
Since any $\mu$ that satisfies Assumption~\ref{Ass: Ellipticity} is always $2$-elliptic it follows 
that $\delta (\mu) > 0$. Consequently, there always exists an open interval $(p_0^{\prime} , p_0)$ for some $p_0 \in (2 , \infty]$, 
such that $\mu$ is $p$-elliptic for all $p \in (p_0^{\prime} , p_0)$. On the other hand,~\cite[Cor.~5.16]{Carbonaro_Dragicevic} states that
\begin{align*}
 [2 , \infty) \ni p \mapsto \Delta_p (\mu)
\end{align*}
is Lipschitz continuous and decreasing. 
Consequently, there is a unique $p_0 \in (2,\infty]$ such that 
$\mu$ is $p$-elliptic \textit{if and only if} $p \in (p_0^{\prime} , p_0)$.

\begin{notation}
For a matrix $\mu$ subject to Assumption~\ref{Ass: Ellipticity} denote by $p_0 (\mu) \in (2 , \infty]$ the unique number such that $\mu$ is $p$-elliptic if and only if $p \in (p_0 (\mu)^{\prime} , p_0 (\mu))$.
\end{notation}

Next, we mention how to quantify a lower bound 
on~\eqref{Eq: Lp dissipativity recasted into p-ellipticity} by means of the
$p$-ellipticity constant. 
Let $p \in (1 , \infty)$ and $\alpha , \beta \in \IR^d$.
If $\alpha^{\prime} = 2 \alpha / p$, then 
\begin{equation} \label{Eq: p-ellipticity estimate}
\Re \Big\langle \mu(x) \Big( \frac{2 \alpha}{p} + \ii \beta \Big) , \frac{2 \alpha}{p^{\prime}} + \ii \beta \Big\rangle = p \Re \Big\langle \mu(x) ( \alpha^{\prime} + \ii \beta ) , \frac{\alpha^{\prime}}{p^{\prime}} + \frac{\ii \beta}{p} \Big\rangle
 \geq \frac{\lambda_p}{2} \bigg( \frac{4 \lvert \alpha \rvert^2}{p} + p \lvert \beta \rvert^2 \bigg)
\end{equation}
for almost every $x \in \Omega$
and any number $0 < \lambda_p \leq \Delta_p (\mu)$. 
In particular, this shows that
\begin{align*}
 \Big\lvert \Big\langle \mu(x) \Big( \frac{2 \alpha}{p} + \ii \beta \Big) , \frac{2 \alpha}{p^{\prime}} + \ii \beta \Big\rangle \Big\rvert 
 \leq C \cdot \Re \Big\langle \mu(x) \Big( \frac{2 \alpha}{p} + \ii \beta \Big) , \frac{2 \alpha}{p^{\prime}} + \ii \beta \Big\rangle,
\end{align*}
for some constant $C > 0$ that only depends on $c^{\bullet}$, $p$ and $\lambda_p$. 
It follows that there exists some angle $\omega \in (0 , \pi / 2)$ that only depends on 
$d$, $c^{\bullet}$, $p$ and $\lambda_p$ such that 
\begin{align}
\label{Eq: p-sector}
 \Big\langle \mu(x) \Big( \frac{2 \alpha}{p} + \ii \beta \Big) , 
        \frac{2 \alpha}{p^{\prime}} + \ii \beta \Big\rangle \in \overline{\Sec_{\omega}}
\end{align}
for almost every $x \in \Omega$ and all $\alpha , \beta \in \IR^d$. 

To conclude this preparatory section, we state a short result that quantifies the size of the imaginary part of $\mu$ by the real part of $\mu$ of a $p$-elliptic matrix, where $p > 2$.

\begin{lemma}
\label{Lem: Control of imaginary part}
Let $\mu$ be $p$-elliptic for some $p > 2$. Then 
\begin{align}
\label{Eq: Control of imaginary part under p-ellipticity condition}
 \lvert \Im (\mu(x)) \rvert_{\Lop(\IR^d)} \leq \frac{(p - 1)^{1 / 2}}{p - 2} \lvert \Re(\mu(x)) \rvert_{\Lop(\IR^d)}
\end{align}
for almost every $x \in \Omega$.
In particular, if $\mu$ is subject to Assumption~\ref{Ass: Ellipticity} then $\mu$ is $p$-elliptic for all $p \in (1 , \infty)$ if and only if $\Im (\mu) = 0$.
\end{lemma}

\begin{proof}
Estimating the right-hand side of~\eqref{Eq: p-ellipticity inequality} from below by zero 
implies that 
\begin{align*}
 \lvert \langle \Im (\mu (x)) \beta , \alpha \rangle \rvert \leq p^{\prime} \bigg( \lvert \Re (\mu (x)) \rvert_{\Lop(\IR^d)} \bigg(\frac{\lvert \alpha \rvert^2}{p^{\prime}} + \frac{\lvert \beta \rvert^2}{p} \bigg) + \lvert \Im(\mu(x)) \rvert_{\Lop(\IR^d)} \frac{\lvert \alpha \rvert \lvert \beta \rvert}{p} \bigg)
\end{align*}
for almost every $x \in \Omega$ and all $\alpha , \beta \in \IR^d$
Taking the supremum over all $\alpha \in \IR^d$ with $\lvert \alpha \rvert = 1$ gives
\begin{align*}
 \lvert \Im (\mu (x)) \beta \rvert \leq p^{\prime} \bigg( \lvert \Re (\mu (x)) \rvert_{\Lop(\IR^d)} \bigg(\frac{1}{p^{\prime}} + \frac{\lvert \beta \rvert^2}{p} \bigg) + \lvert \Im(\mu(x)) \rvert_{\Lop(\IR^d)} \frac{\lvert \beta \rvert}{p} \bigg).
\end{align*}
Let $t > 0$, replace $\beta$ by $t \beta$, divide by $t$, and minimize the right-hand side with respect to the parameter $t$ to obtain
\begin{align*}
 \lvert \Im(\mu(x)) \beta \rvert \leq 2 \lvert \Re(\mu (x)) \rvert_{\Lop(\IR^d)} \bigg(\frac{p^{\prime}}{p} \bigg)^{1 / 2} \lvert \beta \rvert + \frac{p^{\prime}}{p} \lvert \Im(\mu(x)) \rvert_{\Lop(\IR^d)} \lvert \beta \rvert.
\end{align*}
Taking the supremum over all $\beta \in \IR^d$ 
with $\lvert \beta \rvert = 1$ shows~\eqref{Eq: Control of imaginary part under p-ellipticity condition}. 

Notice that for $p \to \infty$ one concludes that $\Im (\mu) = 0$ if $\mu$ is $p$-elliptic for all $p \in (1 , \infty)$. On the other hand, if $\Im (\mu) = 0$ a direct calculation shows that $\mu$ is $p$-elliptic for all $p \in (1 , \infty)$.
\end{proof}

\section{Formulation of the main results}
\label{Sec: Formulation of the main results}

The first main result provides an interval in $(1 , \infty)$ in which one obtains 
resolvent bounds for the operator $A_p$. 

\begin{theorem}
\label{Thm: Analyticity}
Let $\Omega \subset \IR^d$ be open and bounded, where $d \geq 3$.
Let $D \subset \partial \Omega$ be closed and set $N = \partial \Omega \setminus D$.
Adopt Assumption~N and let $\mu \in \L^{\infty} (\Omega ; \IC^{d \times d})$ be 
subject to Assumption~\ref{Ass: Ellipticity}. 
Then for all $p \in (1 , \infty)$ that satisfy
\begin{align}
\label{Eq: p interval}
 \frac{p_0 (\mu) d}{d(p_0 (\mu) - 1) + 2} < p < \frac{p_0 (\mu) d}{d - 2}
\end{align}
the operator $A_p$ is sectorial. 
This means, there exist $\theta \in (\pi / 2 , \pi)$ and $C > 0$ 
such that $\Sec_{\theta} \subset \rho(- A_p)$ and 
\begin{align*}
 \| \lambda (\lambda + A_p)^{-1} f \|_{\L^p (\Omega)} 
\leq C \| f \|_{\L^p(\Omega)} \qquad (\lambda \in \Sec_{\theta} , \; f \in \L^p(\Omega)).
\end{align*}
Finally, given $\gamma_0 > 0$, the constants $\theta$ and $C$ can be chosen to depend only on 
$d$, $p$, $M$, $c_{\bullet}$, $c^{\bullet}$ and $\gamma_0$,
whenever $\Delta_{\max(2, \frac{d-2}{d} p)}(\mu) \geq \gamma_0$ if $p \geq 2$ and 
with $p$ replaced by $p'$ if $p < 2$.
\end{theorem}

We shall prove Theorem~\ref{Thm: Analyticity} in Section~\ref{Sec: From weak reverse Holder estimates to operator bounds}.
This theorem has the following direct corollary.

\begin{corollary} \label{cnumrange302}
Let $\Omega \subset \IR^d$ be open and bounded, where $d \geq 3$.
Let $D \subset \partial \Omega$ be closed and set $N = \partial \Omega \setminus D$.
Adopt Assumption~N and let $\mu \in \L^{\infty} (\Omega ; \IC^{d \times d})$ be 
subject to Assumption~\ref{Ass: Ellipticity}. 
Then for every $p \in (1 , \infty)$ that satisfies condition~\eqref{Eq: p interval} 
the operator $- A_p$ generates a bounded analytic semigroup $(\e^{- t A_p})_{t \geq 0}$ on~$\L^p(\Omega)$.
\end{corollary}

Let $X$ be a Banach space, $- \mathcal{A}$ the generator of an analytic semigroup 
on $X$, $0 < T \leq \infty$ and $1 < q < \infty$. Consider the problem
\begin{align*}
\left\{
\begin{aligned}
 u^{\prime} (t) + \mathcal{A} u (t) &= f(t) \qquad (0 < t < T), \\
 u(0) &= 0.
\end{aligned}
\right.
\end{align*}
We say that $\mathcal{A}$ has maximal $\L^q$-regularity if for every $f \in \L^q(0 , T ; X)$ the unique mild solution
\begin{align*}
 u(t) \coloneqq \int_0^t \e^{- (t - s) \mathcal{A}} f(s) \; \d s \qquad (0 < t < T)
\end{align*}
satisfies $u(t) \in \dom(\mathcal{A})$ for almost every $0 < t < T$ and $\mathcal{A} u , u^{\prime} \in \L^q(0 , T ; X)$. Since maximal $\L^q$-regularity is independent of $q$, we will simply say maximal regularity. 

Under Assumption~N, Egert has shown in~\cite[Thm.~1.3]{egert_kato_Lp} that the $\mathrm{H}^{\infty}$-calculus of $A_p$ is bounded with corresponding $\mathrm{H}^{\infty}$-angle less than $\pi / 2$ whenever $- A_p$ generates a bounded analytic semigroup. Since the boundedness of the $\mathrm{H}^{\infty}$-calculus with angle less than $\pi / 2$ in turn implies that $A_p$ has maximal regularity for every $0 < T \leq \infty$ by~\cite[p.~340]{Kalton_Weis}, we have the following corollary.

\begin{corollary}
\label{Cor: Maximal regularity on Lp}
Let $\Omega \subset \IR^d$ be open and bounded, where $d \geq 3$.
Let $D \subset \partial \Omega$ be closed and set $N = \partial \Omega \setminus D$.
Adopt Assumption~N and let $\mu \in \L^{\infty} (\Omega ; \IC^{d \times d})$ be 
subject to Assumption~\ref{Ass: Ellipticity}. 
Then for every $p \in (1 , \infty)$ that satisfies condition~\eqref{Eq: p interval} and 
for every $0 < T \leq \infty$ the operator $A_p$ has maximal regularity.
\end{corollary}

The next theorem provides optimal estimates for gradients of resolvents.
Hence it gives
optimal estimates of the gradient of the solution to the resolvent equation. In particular, a higher-integrability estimate for right-hand sides in $\L^p (\Omega)$ with $2 < p < 2 + \eps$ is derived. 
In the case $\lambda = 0$ and for Dirichlet boundary conditions this resembles a classical 
result due to Meyers~\cite{Meyers}. 
The result of Meyers essentially states that there exists an $\eps > 0$ such that for all $2 < p < 2 + \eps$ the gradient of a 
solution to an elliptic equation with right-hand side in $\W^{-1 , p} (\Omega)$ lies in 
$\L^p (\Omega)$ as well. 
The theorem below gives a corresponding result for the resolvent problem including quantitative estimates with respect to the resolvent parameter $\lambda$. 
Note that additional geometric assumptions are required in order to obtain the existence of this number $\eps > 0$.

\begin{theorem}
\label{Thm: Gradient estimates for resolvent}
Let $\Omega \subset \IR^d$ be open and bounded, where $d \geq 3$.
Let $D \subset \partial \Omega$ be closed and set $N = \partial \Omega \setminus D$.
Adopt Assumption~N and let $\mu \in \L^{\infty} (\Omega ; \IC^{d \times d})$ be 
subject to Assumption~\ref{Ass: Ellipticity}. 
Let $p \in (1 , \infty)$.
\begin{tabel} 
\item \label{Thm: Gradient estimates for resolvent-1}
If
\begin{align}
\label{Eq: Interval gradient estimate arbitrary domain}
 \frac{p_0 (\mu) d}{d(p_0 (\mu) - 1) + 2} < p \leq 2
\end{align}
then  $\dom(A_p) \subset \W^{1 , p}_D (\Omega)$.
Moreover, given $\gamma_0 > 0$ there exists a constant $C > 0$ such that
\begin{align}
\label{Eq: Gradient estimate for resolvent}
 \lvert \lambda \rvert^{1 / 2} \| \nabla (\lambda + A_p)^{-1} f \|_{\L^p (\Omega ; \IC^d)} 
\leq C \| f \|_{\L^p(\Omega)} 
\qquad (\lambda \in \Sec_{\theta} , \; f \in \L^p (\Omega)) .
\end{align}
Here $C > 0$ depends only on $d$, $p$, $M$, $c_{\bullet}$, $c^{\bullet}$ and $\gamma_0$, 
whenever $\Delta_{\max(2, \frac{d-2}{d} p')}(\mu) \geq \gamma_0$
and $\theta$ denotes the angle as in Theorem~\ref{Thm: Analyticity}. 
\item \label{Thm: Gradient estimates for resolvent-2}
If in addition Assumptions~$\Omega^c$ and~D are satisfied, 
then there exists an $\eps > 0$ such that $\dom(A_p) \subset \W^{1 , p}_D (\Omega)$
for all
\begin{align}
\label{Eq: Interval gradient estimates plump domain}
 2 \leq p < 2 + \eps.
\end{align}
Moreover, there exists a constant $C > 0$ such that~\eqref{Eq: Gradient estimate for resolvent} 
is valid. 
Here $C , \eps > 0$ depend only on $d$, $p$, $\kappa$, $\iota$, $r_0$, $s_0$, $M$, $c_{\bullet}$
and $c^{\bullet}$.
\item \label{Thm: Gradient estimates for resolvent-3}
If in addition Assumptions~$\Omega^c$ and~D are satisfied, 
then there exists an $\eps > 0$ such that 
the operator $\nabla (\lambda + A_p)^{-1} \dive$ extends 
from $\C_c^{\infty} (\Omega ; \IC^d)$ to a bounded operator on 
$\L^p (\Omega ; \IC^d)$ for all
\begin{align*}
 (2 + \eps)^{\prime} < p < 2 + \eps
\end{align*}
and $\lambda \in \Sec_{\theta}$. 
Furthermore, there exists a $C > 0$ such that
\begin{align*}
 \| \nabla (\lambda + A_p)^{-1} \dive f \|_{\L^p (\Omega ; \IC^d)} \leq C \| f \|_{\L^p (\Omega ; \IC^d)} \qquad (\lambda \in \Sec_{\theta} , f \in \L^p (\Omega ; \IC^d)).
\end{align*}
Here $\eps > 0$ is as in~\ref{Thm: Gradient estimates for resolvent-2} and $C > 0$ 
depends only on $d$, $p$, $\kappa$, $\iota$, $r_0$, $s_0$, $M$, $c_{\bullet}$
and $c^{\bullet}$.
\end{tabel}
\end{theorem}

We shall prove Theorem~\ref{Thm: Gradient estimates for resolvent} in Section~\ref{Sec: From weak reverse Holder estimates to operator bounds}.

\begin{rem}
An analogue of Theorem~\ref{Thm: Gradient estimates for resolvent}\ref{Thm: Gradient estimates for resolvent-2} 
and~\ref{Thm: Gradient estimates for resolvent-3} for elliptic systems with $\L^{\infty}$-coefficients 
follows literally by the same lines of proof and it holds also in two space-dimensions.
\end{rem}

The representation of $\e^{- t A_p}$ by the Cauchy integral formula directly leads to the following corollary.

\begin{corollary}
\label{Cor: Gradient estimates}
Let $\Omega \subset \IR^d$ be open and bounded, where $d \geq 3$.
Let $D \subset \partial \Omega$ be closed and set $N = \partial \Omega \setminus D$.
Adopt Assumption~N and let $\mu \in \L^{\infty} (\Omega ; \IC^{d \times d})$ be 
subject to Assumption~\ref{Ass: Ellipticity}. 
Let $p \in \big(\frac{p_0 (\mu) d}{d(p_0 (\mu) - 1) + 2} , 2 \big]$.
Then there exists a constant $C > 0$ such
that
\begin{align*}
 \| \nabla \e^{- t A_p} f \|_{\L^p (\Omega ; \IC^d)} 
\leq C t^{- 1 / 2} \| f \|_{\L^p (\Omega)} 
\qquad (f \in \L^p (\Omega) , \; t > 0) .
\end{align*}

If in addition Assumptions~$\Omega^c$ and~D are satisfied,
the same estimate holds for all $p \in [2, 2 + \varepsilon)$,
where $\varepsilon > 0$ is as in 
Theorem~\ref{Thm: Gradient estimates for resolvent}\ref{Thm: Gradient estimates for resolvent-2}. 
Moreover, there exists a $C > 0$ such that 
\begin{align*}
  \| \nabla \e^{- t A_p} \dive f \|_{\L^p (\Omega ; \IC^d)} 
\leq C t^{- 1} \| f \|_{\L^p (\Omega; \IC^d)} 
\qquad (f \in \L^p (\Omega ; \IC^d) , \; t > 0)
\end{align*}
for all $(2 + \eps)^{\prime} < p < 2 + \eps$.
\end{corollary}

\begin{rem}
Let $\Omega \subset \IR^d$ open, $D \subset \partial \Omega$ closed and $N :=
\partial \Omega \setminus D$ subject to Assumption~N. 
In~\cite{egert_kato_Lp}, Egert proved under the assumptions that $\Omega$ is a $d$-set and $D$ a $(d - 1)$-set
that there exists an $\eps > 0$ such that for all numbers $p$ that 
satisfy~\eqref{Eq: Interval gradient estimates plump domain} 
or~\eqref{Eq: Interval gradient estimate arbitrary domain} 
the operator $A_p^{1 / 2} \colon \W^{1 , p}_D (\Omega) \to \L^p(\Omega)$ is an isomorphism. 
Direct consequences of this isomorphism property are the gradient estimates in 
Corollary~\ref{Cor: Gradient estimates}. 
We emphasise that the gradient estimates in Theorem~\ref{Thm: Gradient estimates for resolvent}\ref{Thm: Gradient estimates for resolvent-1}
are valid for all $p \leq 2$ subject to~\eqref{Eq: Interval gradient estimate arbitrary domain}
without any additional assumptions to $\Omega$ and $D$. 
We note that for the estimates in Theorem~\ref{Thm: Gradient estimates for resolvent}\ref{Thm: Gradient estimates for resolvent-3}
for the case $p \in [2 , 2 + \varepsilon)$ 
neither Assumptions~$\Omega^c$ and~D imply the assumptions made by Egert nor do
the assumptions made by Egert imply the Assumptions~$\Omega^c$ and~D.
Vividly spoken, our Assumptions~$\Omega^c$ and~D allow for the
presence of outward cusps in the Dirichlet part while Egert's
assumptions allow for the presence of inward cusps and slits in the
Dirichlet part. 
Note that under Assumption~N these cusps cannot
accumulate towards the Neumann boundary.
\end{rem}

Our final result concerns the perturbation of real-valued matrices of coefficients $\mu$ by 
complex-valued matrices.
We obtain semigroups which are consistent on $\L^p(\Omega)$ for all $p \in [1,\infty)$.
As a consequence also all resolvent operators are consistent on $\L^p(\Omega)$ for the 
full range $p \in [1,\infty)$.
We emphasise that the semigroups are no longer quasi-contractive for all $p \in [1,\infty)$.

\begin{theorem} \label{tnumrange307}
Let $\Omega \subset \IR^d$ be open and bounded, where $d \geq 3$.
Let $D \subset \partial \Omega$ be closed and set $N = \partial \Omega \setminus D$. 
Assume that $\Omega^c$ is a $d$-set and
adopt Assumptions~N and~P.
Finally let $\mu \in \L^{\infty} (\Omega ; \IR^{d \times d})$ be 
a real-valued coefficient function
subject to Assumption~\ref{Ass: Ellipticity}. 

Then there exists an $\varepsilon > 0$ such that for all (complex valued)
$\nu \in \L^{\infty} (\Omega ; \IC^{d \times d})$ satisfying
$\|\nu\|_{\L^{\infty} (\Omega ; \Lop(\IC^d))} < \varepsilon$
the $\L^2$-realisation $A_2 = - \nabla \cdot (\mu + \nu) \nabla$ is
m-sectorial. Moreover,
for all $p \in (1,\infty)$ the operator $- A_p$ generates an 
analytic $C_0$-semigroup on $\L^p (\Omega)$.
These semigroups also extend to an analytic $C_0$-semigroup on $\L^1 (\Omega)$.
The semigroups have a kernel with Gaussian bounds, that is, 
there are $b,c,\omega > 0$ such that for all $t > 0$ there exists 
a measurable $K_t \colon \Omega \times \Omega \to \IC$ that satisfies
\[
|K_t(x,y)|
\leq c  t^{-d/2}  \e^{-b \frac{|x-y|^2}{t}}  \e^{\omega t}
\]
for all $t > 0$ and $x,y \in \Omega$,
and 
\[
(\e^{-t A_p} u)(x)
= \int_\Omega K_t(x,y) u(y) \; \d y
\]
for all $t > 0$, $p \in (1,\infty)$, $u \in \L^p(\Omega)$ and $x \in \Omega$.
The constants $\varepsilon$, $b$, $c$ and $\omega$ can be chosen to depend only
on $d$, $M$, $c_\bullet$, $c^\bullet$, $\alpha$, $c_0$ and $c_1$. Here $\alpha > 0$ 
is such that $|B(x,r) \cap \Omega^c| \geq \alpha r^d$
for all $x \in \partial \Omega$ and $r \in (0,1]$.
\end{theorem}

We prove Theorem~\ref{tnumrange307} in Section~\ref{Sec: Perturbation of real-valued matrices}.
A combination of Theorem~\ref{tnumrange307} with Lemma~\ref{Lem: Control of imaginary part} establishes the following corollary.

\begin{corollary}
\label{Cor: Perturbation translated to p-ellipticity}
Let $\Omega \subset \IR^d$ be open and bounded, where $d \geq 3$.
Let $D \subset \partial \Omega$ be closed and set 
$N = \partial \Omega \setminus D$. 
Assume that $\Omega^c$ is a $d$-set and adopt Assumptions~N and~P.
Given $c_{\bullet} , c^{\bullet} > 0$ there exists a $p_c > 2$, depending 
only on $c_{\bullet}$, $c^{\bullet}$, $d$, $M$, $\alpha$, $c_0$ and $c_1$, 
such that for every $\mu \in \L^{\infty} (\Omega ; \IC^{d \times d})$, that is
subject to Assumption~\ref{Ass: Ellipticity} with constants $c_\bullet$ and
$c^\bullet$ and satisfies $p_0(\mu) > p_c$,
the operator $-A_p$ generates an analytic $C_0$-semigroup for all $p \in (1 , \infty)$,
which is consistent with an analytic $C_0$-semigroup on $\L^1(\Omega)$.
Here $\alpha > 0$ is such that $|B(x,r) \cap \Omega^c| \geq \alpha r^d$
for all $x \in \partial \Omega$ and $r \in (0,1]$.
\end{corollary}

\section{Weak reverse H\"older estimates for resolvent equations}
\label{Sec: Reverse Holder estimates for the resolvent equations}

Throughout this section 
let $\Omega \subset \IR^d$ be open and bounded, where $d \geq 3$.
Let $D \subset \partial \Omega$ be closed and set $N = \partial \Omega \setminus D$.
Adopt Assumption~N and let $\mu \in \L^{\infty} (\Omega ; \IC^{d \times d})$ be 
subject to Assumption~\ref{Ass: Ellipticity}. 
In order to prove Theorem~\ref{Thm: Analyticity}, we show the validity of weak reverse
H\"older estimates for the resolvent. 
More precisely, let $u \in \W^{1 ,2}_D (\Omega)$ be a solution to
\begin{align*}
 \lambda \int_{\Omega} u  \overline v  \; \d x + \int_{\Omega} \langle \mu \nabla u , \nabla v \rangle \; \d x = 0,
\end{align*}
for all $v \in \W^{1 , 2}_D (\Omega)$ with $\supp(v) \subset \overline{B(x_0 , 2 r)}$. Then for any $p \in (2 , d p_0 (\mu) / (d - 2))$, we prove that there exist constants $C > 0$ and $c \in (0 , 1)$ such that
\begin{align*}
 \bigg( \frac{1}{r^d} \int_{\O (x_0 , c r)} \lvert u \rvert^p \; \d x \bigg)^{1 / p} 
\leq C \bigg( \frac{1}{r^d} \int_{\O (x_0 , 2 r)} \lvert u \rvert^2 \; \d x \bigg)^{1 / 2}.
\end{align*}
For interior balls and in the case $\lambda = 0$ this inequality was proven by Dindo\v{s} and Pipher in~\cite{Dindos_Pipher} which follows the ideas of Cialdea and Maz'ya~\cite{cialdea/mazya}. 
We choose, however, a slightly different approach that allows to treat balls centred at the boundary. 

To begin with, we introduce the following cut-off function. 
For all $0 < \delta < L$ define $\lvert \; \cdot \; \rvert_{\delta , L} \colon \IR \to (0,\infty)$
by $\lvert x \rvert_{\delta , L} = \delta \vee |x| \wedge L$. 
So 
\begin{align*}
 \lvert x \rvert_{\delta , L} 
= \begin{cases} \delta & \text{if } \lvert x \rvert \leq \delta, \\ 
\lvert x \rvert & \text{if } \delta < \lvert x \rvert < L, \\ 
L &\text{if } \lvert x \rvert \geq L. \end{cases}
\end{align*}
In order to prove the weak reverse H\"older estimate, we will test the resolvent equation with the testfunction
\[
 v \coloneqq \eta^2 \lvert u\rvert_{\delta , L}^{p - 2} u,
\]
where $\eta$ is a suitable cut-off function. To do so, we need the following lemma.

\begin{lemma}
\label{Lem: Testfunction}
If $p \in [1,\infty)$, $0 < \delta < L$ and 
$u \in \W^{1 , 2}_D (\Omega)$, then 
$w := \lvert u \rvert_{\delta , L}^{p - 2} u \in \W^{1 , 2}_D (\Omega)$.
\end{lemma}

\begin{proof}
There exists a sequence $(u_n)_{n \in \IN}$ in $\C^{\infty}_D (\O)$ such that 
$u_n \to u$ as $n \to \infty$ in $\W^{1 , 2} (\O)$. 
Passing to a subsequence if necessary, we may without loss of generality
assume that $u_n \to u$ and $\nabla u_n \to \nabla u$ almost everywhere
as $n \to \infty$. 
For all $n \in \IN$ define $w_n \coloneqq \lvert u_n \rvert_{\delta , L}^{p - 2} u_n$.
Then $\dist(\supp(w_n) , D) > 0$.
If $n \in \IN$, then
\begin{align}
\label{Eq: L2 convergence}
\begin{aligned}
 \| w_n - w \|_{\L^2(\O)} 
&\leq \| (\lvert u_n \rvert^{p - 2}_{\delta , L} - \lvert u \rvert^{p - 2}_{\delta , L}) u \|_{\L^2(\O)}
  + \| \lvert u_n \rvert^{p - 2}_{\delta , L} (u_n - u) \|_{\L^2(\O)}  \\
 &\leq \| (\lvert u_n \rvert^{p - 2}_{\delta , L} - \lvert u \rvert^{p - 2}_{\delta , L}) u \|_{\L^2(\O)}
 + \max\{ \delta^{p - 2} , L^{p - 2} \} \| u_n - u \|_{\L^2(\O)} .
\end{aligned}
\end{align}
Hence $\lim_{n \to \infty} w_n = w$ in $\L^2(\Omega)$ by the 
$\L^2$-convergence of $(u_n)_{n \in \IN}$ to
$u$  and the dominated convergence theorem. 

Notice that $w$ is weakly differentiable with weak derivative
\begin{align*}
 \partial_j w = \lvert u \rvert_{\delta , L}^{p - 2} \partial_j u + (p - 2) \chi_{\{\delta < \lvert u \rvert < L\}} \lvert u \rvert^{p - 4} u \Re(\overline{u} \partial_j u).
\end{align*}
The same formula is valid for $\partial_j w_n$ with $u$ replaced by $u_n$. Write shortly $\chi_n \coloneqq \chi_{\{ \delta < \lvert u_n \rvert < L \}}$ and $\chi \coloneqq \chi_{\{ \delta < \lvert u \rvert < L \}}$ in the following.
The convergence of the sequence 
$( \lvert u_n \rvert^{p - 2}_{\delta , L} \partial_j u_n )_{n \in \IN}$ 
to $\lvert u \rvert^{p - 2}_{\delta , L} \partial_j u$ in $\L^2 (\Omega)$ follows as 
in~\eqref{Eq: L2 convergence}.
Hence it remains to estimate
\begin{align*}
 \| \chi_n \lvert u_n \rvert^{p - 4} u_n \Re(\overline{u_n} \partial_j u_n) - \chi \lvert u \rvert^{p - 4} u \Re(\overline{u} \partial_j u) \|_{\L^2 (\Omega)}
\end{align*}
This is bounded, using the triangle inequality, by
\begin{align*}
L^2 \max\{ \delta^{p - 4} , L^{p - 4} \} \| \partial_j u_n - \partial_j u \|_{\L^2 (\Omega)} + \| \chi_n \lvert u_n \rvert^{p - 4} u_n \Re(\overline{u_n} \partial_j u) - \chi \lvert u \rvert^{p - 4} u \Re(\overline{u} \partial_j u) \|_{\L^2 (\O)}.
\end{align*}
The first term converges to zero as $n \to \infty$ by the 
$\L^2$-convergence of $(\nabla u_n)_{n \in \IN}$ to $\nabla u$.
For the second term, introduce the sets
\begin{align*}
 C_0 \coloneqq \{ x \in \O : \lvert u (x) \rvert \in (\delta, L) \},& \quad C_1 \coloneqq \{ x \in \O : \lvert u (x) \rvert \in [0 , \delta) \cup (L , \infty] \}, \\
 C_2 \coloneqq \{ x \in \O : \lvert u (x) \rvert = \delta \},& \quad C_3 \coloneqq \{ x \in \O : \lvert u(x) \rvert = L \}
\end{align*}
and decompose the domain of integration to deduce the inequality
\begin{align*}
 & \| \chi_n \lvert u_n \rvert^{p - 4} u_n \Re(\overline{u_n} \partial_j u) - \chi \lvert u \rvert^{p - 4} u \Re(\overline{u} \partial_j u) \|_{\L^2 (\O)} \\
 &\leq \sum_{k = 1}^3 \| \chi_n \lvert u_n \rvert^{p - 4} u_n \Re(\overline{u_n} \partial_j u) \|_{\L^2 (C_k)}
 + \| \chi_n \lvert u_n \rvert^{p - 4} u_n \Re(\overline{u_n} \partial_j u) - \chi \lvert u \rvert^{p - 4} u \Re(\overline{u} \partial_j u) \|_{\L^2 (C_0)}.
\end{align*}
Since $\lim_{n \to \infty} \chi_n \lvert u_n \rvert^{p - 4} u_n \Re(\overline{u_n} \partial_j u) =
\chi \lvert u \rvert^{p - 4} u \Re(\overline{u} \partial_j u)$
almost everywhere on $C_0$, the last term tends to zero for
$n \to \infty$ by the dominated convergence theorem.
Similarly the integral on $C_1$ tends to zero.
Next, if $C_2$ is not a set of measure zero,
then $\partial_j \lvert u \rvert = 0$ on $C_2$ almost everywhere
by~\cite[Lem.~7.7]{Gilbarg_Trudinger}. 
This implies that $\Re( \overline{u} \partial_j u) = 0$ on $C_2$ and hence
\begin{align*}
 \Re(\overline{u_n} \partial_j u) = \Re \bigg( \frac{\overline{u_n}}{\overline{u}} \overline{u} \partial_j u \bigg) = - \Im \bigg( \frac{\overline{u_n - u}}{\overline{u}} \bigg) \Im (\overline{u} \partial_j u) \quad \text{on} \quad C_2.
\end{align*}
The same holds on $C_3$.
Now, the dominated convergence theorem implies that the remaining terms on the right-hand side converge to zero. 

We proved that $\lim_{n \to \infty} w_n = w$ in $\W^{1 , 2}(\Omega)$.
It remains to prove that 
$w_n \in \W^{1 , 2}_D (\Omega)$ for all $n \in \IN$.
Since $\W^{1 , 2}_D (\O)$ is a Banach space this implies that $w \in \W^{1 , 2}_D (\O)$.

Let $n \in \IN$.
By definition of the space $\C_D^{\infty} (\O)$ there exists a
$\varphi \in \C_c^{\infty} (\IR^d)$ such that $\supp(\varphi_n) \cap D = \emptyset$ and 
$u_n = \varphi|_\Omega$.
Let $(\rho_k)_{k \in \IN}$ be a standard approximation of the identity.
For all $k \in \IN$ define 
\begin{align*}
 f_k \coloneqq 
(\lvert \varphi \rvert_{\delta , L}^{p - 2} \varphi ) * \rho_k. 
\end{align*}
Then $f_k \in \C_c^{\infty} (\IR^d)$ and if $k$ is large enough then 
$\supp(f_k) \cap D = \emptyset$.
So $f_k|_\Omega \in \W^{1 , 2}_D (\Omega)$ if $k$ is large enough.
Since $\lim_{k \to \infty} f_k|_{\O} = \lvert u_n \rvert_{\delta , L}^{p - 2} u_n = w_n$ 
one deduces that $w_n \in \W^{1 , 2}_D (\Omega)$.
The proof is complete.
\end{proof}

Now, we are in the position to prove the validity of the weak reverse H\"older estimates of the resolvent problem.

\begin{theorem}
\label{Thm: Reverse Holder}
Let $\Omega \subset \IR^d$ be open and bounded, where $d \geq 3$.
Let $D \subset \partial \Omega$ be closed and set $N = \partial \Omega \setminus D$.
Adopt Assumption~N and let $\mu \in \L^{\infty} (\Omega ; \IC^{d \times d})$ be 
subject to Assumption~\ref{Ass: Ellipticity}. 
If $p \in [2 ,p_0 (\mu))$, then there exists a $\theta \in (\pi / 2 , \pi)$ such that for all
$\lambda \in \Sec_{\theta}$, $x_0 \in \overline{\Omega}$ and 
$r \in (0 , M / 4)$ the following is valid. 
Let $u \in \W^{1 , 2}_D (\O)$ with
\begin{align*}
 \lambda \int_{\O} u \overline{v} \; \mathrm{d} x + \mathfrak{t} [u , v] = 0
\end{align*}
for all $v \in \W^{1 , 2}_D (\O)$ with $\supp(v) \subset \overline{B(x_0 , 2 r)}$.
Then there exist $\mathcal{C} > 0$ and $c \in (0 , 1)$ such that 
$u \in \L^{p d / (d - 2)} (\Omega (x_0 , c r))$ and 
\begin{align*}
 \bigg( \frac{1}{r^d} \int_{\O (x_0 , c r)} \lvert u \rvert^{\frac{p d}{d - 2}} \; \mathrm{d} x \bigg)^{\frac{d - 2}{p d}} \leq \mathcal{C} \bigg( \frac{1}{r^d} \int_{\O (x_0 , 2 r)} \lvert u \rvert^2 \; \mathrm{d} x \bigg)^{\frac{1}{2}} .
\end{align*}
For given $\gamma_0 > 0$ the constants $\mathcal{C}$ and $\theta$ 
can be chosen to depend only on $d$, $p$, $M$, $c_{\bullet}$, $c^{\bullet}$ and $\gamma_0$, 
as long as $\Delta_p (\mu) \geq \gamma_0$ and $c$ can be chosen to depend only on $d$, $p$ and $M$.
\end{theorem}

\begin{proof}
Fix $\tau \in \C_c^{\infty} (\IR^d)$ such that 
$\tau|_{B(0 , 1)} = \one$, $0 \leq \tau \leq 1$ and $\supp \tau \subset B(0 , 2)$.
One can arrange that $\| \nabla \tau \|_{\L^{\infty}} \leq 2$.
Let $x_0 \in \R^d$ and  $r > 0$.
Define $\eta \in \C_c^{\infty} (\IR^d)$ by $\eta(x) = \tau( r^{-1}(x - x_0) )$.
Then $\eta \in \C_c^{\infty} (B(x_0 , 2 r))$, $0 \leq \eta \leq 1$,
$\eta|_{B(x_0 , r)} = \one$ and $\| \nabla \eta \|_{\L^{\infty}} \leq 2 / r$.
Let $0 < \delta < L$.
By Lemmas~\ref{Lem: Multiplication by smooth functions} and~\ref{Lem: Testfunction} the testfunction
\begin{align*}
 v \coloneqq \eta^2 \lvert u \rvert_{\delta , L}^{p - 2} u
\end{align*}
satisfies $v \in \W^{1 ,  2}_D (\O)$. Next,
define $\varphi \coloneqq \lvert u \rvert^{\frac{p - 2}{2}} u$ and
$\O_{\delta , L} \coloneqq \{ \delta < \lvert u \rvert < L \}$.
Then $\lvert \varphi \rvert^{\frac{p - 2}{p}} \varphi = \lvert u \rvert^{p - 2} u$.
On the set $\O_{\delta , L}$ one establishes with~\cite[Lem.~7.6]{Gilbarg_Trudinger}
that 
\[ \nabla (\lvert \varphi \rvert^{\frac{p - 2}{p}} \varphi) 
 = \nabla (\lvert u \rvert^{p - 2} u)
 = u \nabla \lvert u \rvert^{p - 2} + \lvert u \rvert^{p - 2} \nabla u 
= u \nabla \lvert u \rvert_{\delta , L}^{p - 2} + \lvert u \rvert_{\delta , L}^{p - 2} \nabla u 
= \nabla (\lvert u \rvert_{\delta , L}^{p - 2} u).
\]
Consequently, $\nabla v = \nabla (\eta^2 \lvert \varphi \rvert^{\frac{p - 2}{p}} \varphi)$.
Using~\cite[Cor.~7.7]{Gilbarg_Trudinger}, the definition of $\varphi$ and the product rule one deduces that 
\begin{align*}
 \mathfrak{t} [u , v] 
&= \int_{\Omega} \langle \mu \nabla u , \nabla (\eta^2 \lvert u \rvert_{\delta , L}^{p - 2} u ) \rangle \; \mathrm{d} x \\
 &= \delta^{p - 2} \int_{ \{ \lvert u \rvert \leq \delta \}} 
    \eta^2 \langle \mu \nabla u , \nabla u \rangle \; \mathrm{d} x 
  + 2 \delta^{p - 2} \int_{ \{ \lvert u \rvert \leq \delta \}} \eta \langle \mu \nabla u , u \nabla \eta  \rangle \; \mathrm{d} x \\
 &\qquad+ L^{p - 2} \int_{ \{ \lvert u \rvert \geq L \}} 
    \eta^2 \langle \mu \nabla u , \nabla u \rangle \; \mathrm{d} x 
  + 2 L^{p - 2} \int_{ \{ \lvert u \rvert \geq L \}} 
    \eta \langle \mu \nabla u , u \nabla \eta \rangle \; \mathrm{d} x \\
 &\qquad+ \int_{\O_{\delta , L}} \langle \mu \nabla (\lvert \varphi \rvert^{- \frac{p - 2}{p}} \varphi) , \nabla (\eta^2 \lvert \varphi \rvert^{\frac{p - 2}{p}} \varphi ) \rangle \; \mathrm{d} x
\end{align*}
Introduce the functions $\Phi \coloneqq \Re(\overline{\sgn\varphi} \nabla \varphi)$ 
and $\Psi \coloneqq \Im (\overline{\sgn\varphi}  \nabla \varphi)$, where
$\sgn z = \frac{z}{|z|}$ if $z \in \IC \setminus \{ 0 \} $ and $\sgn 0 = 0$.
Then
\begin{align}
\label{Eq: Derivative of absolute value}
 \nabla \lvert \varphi \rvert = \Phi.
\end{align}
Relying only on the definitions of $\Phi$ and $\Psi$, the product rule and~\eqref{Eq: Derivative of absolute value} one calculates 
\begin{align*}
 &\int_{\O_{\delta , L}} \langle \mu \nabla (\lvert \varphi \rvert^{- \frac{p - 2}{p}} \varphi) , \nabla (\eta^2 \lvert \varphi \rvert^{\frac{p - 2}{p}} \varphi ) \rangle \; \mathrm{d} x \\
 =& \int_{\O_{\delta , L}} \eta^2 \Big\langle \mu \Big( \frac{2 \Phi}{p} + \ii \Psi \Big) , \Big( \frac{2 \Phi}{p^{\prime}} + \ii \Psi \Big) \Big\rangle \; \mathrm{d} x 
   + \frac{4}{p} \int_{\O_{\delta , L}} \eta \langle \mu \Phi , \lvert \varphi \rvert \nabla \eta \rangle \; \mathrm{d} x \\
 & \hspace{10cm}+ 2 \ii \int_{\O_{\delta , L}} \eta \langle \mu \Psi , \lvert \varphi \rvert \nabla \eta \rangle \; \mathrm{d} x.
\end{align*}
Since $\lambda \int_\Omega u \overline v \; \d x + \mathfrak{t} [u , v] = 0$, 
a combination of the previous calculations and a rearrangement of the terms gives
\begin{align}
\label{Eq: Localized equation with appropriate testfunction}
\begin{aligned}
 & \lambda \int_{\O} \eta^2 \lvert u \rvert_{\delta , L}^{p - 2} \lvert u \rvert^2  \; \mathrm{d} x 
   + \delta^{p - 2} \int_{ \{ \lvert u \rvert \leq \delta \}} \eta^2 \langle \mu \nabla u , \nabla u \rangle \; \mathrm{d} x \\
 &\qquad + L^{p - 2} \int_{ \{ \lvert u \rvert \geq L \}} \eta^2 \langle \mu \nabla u , \nabla u \rangle \; \mathrm{d} x 
+ \int_{\O_{\delta , L}} \eta^2 \Big\langle \mu \Big( \frac{2 \Phi}{p} + \ii \Psi \Big) , \Big( \frac{2 \Phi}{p^{\prime}} + \ii \Psi \Big) \Big\rangle \; \mathrm{d} x \\
 =& - \frac{4}{p} \int_{\O_{\delta , L}} \eta \langle \mu \Phi , \lvert \varphi \rvert \nabla \eta \rangle \; \mathrm{d} x 
    - 2 \ii \int_{\O_{\delta , L}} \eta \langle \mu \Psi , \lvert \varphi \rvert \nabla \eta \rangle \; \mathrm{d} x \\
 &\qquad- 2 \delta^{p - 2} \int_{ \{ \lvert u \rvert \leq \delta \}} \eta \langle \mu \nabla u , u \nabla \eta \rangle \; \mathrm{d} x 
    - 2 L^{p - 2} \int_{ \{ \lvert u \rvert \geq L \}} \eta \langle \mu \nabla u , u \nabla \eta \rangle \; \mathrm{d} x.
\end{aligned}
\end{align}
We next determine the angle $\theta$ for the sector $\Sec_{\theta}$ for~$\lambda$.
Ellipticity of $\mu$ gives
the estimate
$|\Im \langle \mu(x) \xi , \xi \rangle| 
\leq c^{\bullet} |\xi|^2
\leq \frac{c^{\bullet}}{c_{\bullet}} \Re \langle \mu(x) \xi , \xi \rangle$
for all $\xi \in \IC^d$ and almost every $x \in \Omega$.
Define 
\begin{align*}
 \theta_0 \coloneqq \arctan \bigg( \frac{2 c^{\bullet}}{c_{\bullet}} \bigg) 
\in ( 0 , \tfrac{\pi}{2} ) .
\end{align*}
Then $\langle \mu(x) \xi , \xi \rangle \in \overline{\Sec_{\theta_0}}$ 
for all $\xi \in \IC^d$ and almost every $x \in \Omega$.
Hence 
\begin{align*}
 \delta^{p - 2} \int_{ \{ \lvert u \rvert \leq \delta \}} \eta^2 \langle \mu \nabla u , \nabla u \rangle \; \mathrm{d} x \in \overline{ \Sec_{\theta_0} }
\quad \text{and} \quad 
L^{p - 2} \int_{ \{ \lvert u \rvert \geq L \}} \eta^2 \langle \mu \nabla u , \nabla u \rangle \; \mathrm{d} x \in \overline{ \Sec_{\theta_0} }.
\end{align*}
Furthermore, by virtue of~\eqref{Eq: p-sector} there exists an $\omega \in (0 , \frac{\pi}{2} )$,
depending only on $c^{\bullet}$, $p$ and $\gamma_0$, such that
\begin{align*}
 \int_{\O_{\delta , L}} \eta^2 \Big\langle \mu \Big( \frac{2 \Phi}{p} + \ii \Psi \Big) , \Big( \frac{2 \Phi}{p^{\prime}} + \ii \Psi \Big) \Big\rangle \; \mathrm{d} x \in \overline{ \Sec_{\omega} }.
\end{align*}
Choose 
\begin{align*}
 \theta \coloneqq \frac{3 \pi}{4} - \frac{\max\{ \theta_0 , \omega \}}{2},
\end{align*}
then $\theta > \pi / 2$ and
\begin{align*}
 \theta + \max\{ \theta_0 , \omega \} = \frac{3 \pi}{4} + \frac{\max\{ \theta_0 , \omega \}}{2} < \pi.
\end{align*}
Notice that this implies that there exists a $C_{\theta , \omega} > 0$ such that 
\begin{align*}
\lvert z_2 \rvert + \lvert z_3 \rvert + \lvert z_4 \rvert 
\leq \lvert z_1 \rvert + \lvert z_2 \rvert + \lvert z_3 \rvert + \lvert z_4 \rvert 
\leq C_{\theta , \omega} \lvert z_1 + z_2 + z_3 + z_4 \rvert 
\end{align*}
for all $z_1 \in \overline{ \Sec_{\theta} } $ and 
$z_2,z_3,z_4 \in \overline{ \Sec_{\max(\theta_0, \omega)} } $.
Hence if $\lambda \in \Sec_{\theta}$, then
\begin{align*}
 &\quad \delta^{p - 2} \Big\lvert \int_{ \{ \lvert u \rvert \leq \delta \}} \eta^2 \langle \mu \nabla u , \nabla u \rangle \; \mathrm{d} x \Big\rvert \\*
 &\qquad\quad + L^{p - 2} \Big\lvert \int_{ \{ \lvert u \rvert \geq L \}} \eta^2 \langle \mu \nabla u , \nabla u \rangle \; \mathrm{d} x \Big\rvert 
+ \Big\lvert \int_{\O_{\delta , L}} \eta^2 \Big\langle \mu \Big( \frac{2 \Phi}{p} + \ii \Psi \Big) , \Big( \frac{2 \Phi}{p^{\prime}} + \ii \Psi \Big) \Big\rangle \; \mathrm{d} x \Big\rvert \\
 &\leq C_{\theta , \omega} \Big\lvert \lambda \int_{\O} \eta^2 \lvert u \rvert_{\delta , L}^{p - 2} \lvert u \rvert^2  \; \mathrm{d} x 
+ \delta^{p - 2} \int_{ \{ \lvert u \rvert \leq \delta \}} \eta^2 \langle \mu \nabla u , \nabla u \rangle \; \mathrm{d} x \\*
 &\qquad\quad + L^{p - 2} \int_{ \{ \lvert u \rvert \geq L \}} \eta^2 \langle \mu \nabla u , \nabla u \rangle \; \mathrm{d} x 
+ \int_{\O_{\delta , L}} \eta^2 \Big\langle \mu \Big( \frac{2 \Phi}{p} + \ii \Psi \Big) , \Big( \frac{2 \Phi}{p^{\prime}} + \ii \Psi \Big) \Big\rangle \; \mathrm{d} x \Big\rvert .
\end{align*}
Using the ellipticity of $\mu$ together with~\eqref{Eq: p-ellipticity estimate} 
and the condition $\Delta_p(\mu) \geq \gamma_0$
on the left-hand side and~\eqref{Eq: Localized equation with appropriate testfunction}
 on the right-hand side yields
\begin{align}
\label{Eq: Estimate after elliptic estimate}
\begin{aligned}
 &\quad c_{\bullet} \delta^{p - 2} \int_{ \{ \lvert u \rvert \leq \delta \}} \eta^2 \lvert \nabla u \rvert^2 \; \mathrm{d} x 
+ c_{\bullet} L^{p - 2} \int_{ \{ \lvert u \rvert \geq L \}} \eta^2 \lvert \nabla u \rvert^2 \; \mathrm{d} x \\
 &\qquad\quad+ \frac{2 \gamma_0}{p} \int_{\O_{\delta , L}} \eta^2 \lvert \Phi \rvert^2 \; \mathrm{d} x 
+ \frac{p \gamma_0}{2} \int_{\O_{\delta , L}} \eta^2 \lvert \Psi \rvert^2 \; \mathrm{d} x \\
 &\leq c^{\bullet} C_{\theta , \omega} \bigg\{ \frac{4}{p} \int_{\O_{\delta , L}} \lvert \eta \Phi \rvert \lvert \varphi \nabla \eta \rvert \; \mathrm{d} x 
+ 2 \int_{\O_{\delta , L}} \lvert \eta \Psi \rvert \lvert \varphi \nabla \eta \rvert \; \mathrm{d} x \\
 &\qquad\quad+ 2 \delta^{p - 2} \int_{ \{ \lvert u \rvert \leq \delta \}} \lvert \eta \nabla u \rvert \lvert u \nabla \eta \rvert \; \mathrm{d} x + 2 L^{p - 2} \int_{ \{ \lvert u \rvert \geq L \}} \lvert \eta \nabla u \rvert \lvert u \nabla \eta \rvert \; \mathrm{d} x \bigg\}.
\end{aligned}
\end{align}
Use the properties of the cut-off function $\eta$ as well as Young's inequality to derive the estimates
\begin{align*}
 \frac{4}{p} c^{\bullet} C_{\theta , \omega} \int_{\O_{\delta , L}} 
   \lvert \eta \Phi \rvert \lvert \varphi \nabla \eta \rvert \; \mathrm{d} x 
&\leq \frac{\gamma_0}{p} \int_{\O_{\delta , L}} \eta^2 \lvert \Phi \rvert^2 \; \mathrm{d} x + \frac{4 C_{\theta , \omega}^2}{p \gamma_0 r^2} c^{\bullet 2} \int_{\O_{\delta , L} \cap B(x_0 , 2 r)} \lvert \varphi \rvert^2 \; \mathrm{d} x, \\
 2 c^{\bullet} C_{\theta , \omega} \int_{\O_{\delta , L}} \lvert \eta \Psi \rvert \lvert \varphi \nabla \eta \rvert \; \mathrm{d} x 
&\leq \frac{p \gamma_0}{2} \int_{\O_{\delta , L}} \eta^2 \lvert \Psi \rvert^2 \; \mathrm{d} x + \frac{2 C_{\theta , \omega}^2}{p \gamma_0 r^2} c^{\bullet 2} \int_{\O_{\delta , L} \cap B (x_0 , 2 r)} \lvert \varphi \rvert^2 \; \mathrm{d} x,\\
 2 \delta^{p - 2} c^{\bullet} C_{\theta , \omega} \int_{ \{ \lvert u \rvert \leq \delta \}} \lvert \eta \nabla u \rvert \lvert u \nabla \eta \rvert \; \mathrm{d} x 
&\leq \frac{\delta^{p - 2} c_{\bullet}}{2} \int_{ \{ \lvert u \rvert \leq \delta \}} \eta^2 \lvert \nabla u \rvert^2 \; \mathrm{d} x \displaybreak[0]\\
 &\qquad\qquad+ \frac{2 \delta^{p - 2} c^{\bullet 2} C_{\theta , \omega}^2}{c_{\bullet} r^2} \int_{ \{ \lvert u \rvert \leq \delta \} \cap B(x_0 , 2 r)} \lvert u \rvert^2 \; \mathrm{d} x\\
\intertext{and}
 2 L^{p - 2} c^{\bullet} C_{\theta , \omega} \int_{ \{ \lvert u \rvert \geq L \}} \lvert \eta \nabla u \rvert \lvert u \nabla \eta \rvert \; \mathrm{d} x 
&\leq \frac{c_{\bullet} L^{p - 2}}{2} \int_{ \{ \lvert u \rvert \geq L \}} \eta^2 \lvert \nabla u \rvert^2 \; \mathrm{d} x \\
 &\qquad\qquad+ \frac{2 L^{p - 2} c^{\bullet 2} C_{\theta , \omega}^2}{c_{\bullet} r^2} \int_{ \{ \lvert u \rvert \geq L \} \cap B(x_0 , 2 r)} \lvert u \rvert^2 \; \mathrm{d} x.
\end{align*}
Combining these estimates with~\eqref{Eq: Estimate after elliptic estimate}, 
while incorporating the fact that $\eta \equiv 1$ on $B(x_0 , r)$ 
and rearranging yields
\begin{align}
\label{Eq: Cutoff Moser type estimate}
\begin{aligned}
 & c_{\bullet} \delta^{p - 2} \int_{\{ \lvert u \rvert \leq \delta \} \cap B(x_0 , r) } \lvert \nabla u \rvert^2 \; \mathrm{d} x 
 + c_{\bullet} L^{p - 2} \int_{\{ \lvert u \rvert \geq L \} \cap B(x_0 , r)} \lvert \nabla u \rvert^2 \; \mathrm{d} x 
+ \frac{2 \gamma_0}{p} \int_{\O_{\delta , L} \cap B(x_0 , r)} \lvert \Phi \rvert^2 \; \mathrm{d} x  \\
 &\leq \frac{4 C_{\theta , \omega}^2 c^{\bullet 2}}{r^2} \bigg\{ \frac{3}{p \gamma_0} \int_{\O_{\delta , L} \cap B(x_0 , 2 r)} \lvert \varphi \rvert^2 \; \mathrm{d} x \\
 &\hspace*{20mm} \qquad+ \frac{\delta^{p - 2}}{c_{\bullet}} \int_{ \{ \lvert u \rvert \leq \delta \} \cap B(x_0 , 2 r)} \lvert u \rvert^2 \; \mathrm{d} x + \frac{L^{p - 2}}{c_{\bullet}} \int_{ \{ \lvert u \rvert \geq L \} \cap B(x_0 , 2 r)} \lvert u \rvert^2 \; \mathrm{d} x \bigg\}.
\end{aligned}
\end{align}
Next, define analogously to the function $\varphi$ the function $\varphi_{\delta , L} := \lvert u \rvert_{\delta , L}^{\frac{p - 2}{2}} u$.
Then Lemmas~\ref{Lem: Testfunction} and~\ref{Lem: Modulus lemma} imply that 
$\lvert \varphi_{\delta , L} \rvert \in \W^{1 , 2}_D (\Omega)$.
Define $\alpha \coloneqq M^2 \sqrt{d}$ and recall that $2^* = 2d / (d - 2)$.
An application of the local embedding 
Lemma~\ref{Lem: Local Sobolev embedding} 
to $\lvert \varphi_{\delta , L} \rvert$ together with~\eqref{Eq: Derivative of absolute value} gives 
\begin{align*}
 &\quad 
\bigg( \frac{1}{r^d} \int_{\O (x_0 , r / \alpha)} \lvert u \rvert_{\delta , L}^{\frac{2^*}{2} (p - 2)} \lvert u \rvert^{2^*} \; \d x \bigg)^{\frac{1}{2^*}} 
= \bigg( \frac{1}{r^d} \int_{\O (x_0 , r / \alpha)} \lvert \varphi_{\delta , L} \rvert^{2^*} \; \d x \bigg)^{\frac{1}{2^*}} \\
 &\leq C_{\mathrm{Sob}} \bigg\{ r \bigg( \frac{1}{r^d} \int_{\O_{\delta , L} \cap B (x_0 , r)} \lvert \Phi \rvert^2 \; \d x \bigg)^{\frac{1}{2}} + r \bigg( \frac{1}{r^d} \int_{\{ \lvert u \rvert \leq \delta \} \cap B (x_0 , r)} \delta^{p - 2} \lvert \nabla u \rvert^2 \; \d x \bigg)^{\frac{1}{2}} \\
 &\qquad\quad + r \bigg( \frac{1}{r^d} \int_{\{ \lvert u \rvert \geq L \} \cap B (x_0 , r)} L^{p - 2} \lvert \nabla u \rvert^2 \; \d x \bigg)^{\frac{1}{2}} + \alpha \bigg( \frac{1}{r^d} \int_{\O (x_0 , r)} \lvert u \rvert_{\delta , L}^{p - 2} \lvert u \rvert^2 \; \d x \bigg)^{\frac{1}{2}} \bigg\}
,
\end{align*}
where we used that $|\nabla |u| \, | \leq |\nabla u|$. 
Apply~\eqref{Eq: Cutoff Moser type estimate} to the first three terms of the right hand side
and split the integral in the fourth term in three parts.
Then
\begin{align}
\label{Eq: Almost the final inequality}
\begin{aligned}
 &\bigg( \frac{1}{r^d} \int_{\O (x_0 , r / \alpha)} \lvert u \rvert_{\delta , L}^{\frac{2^*}{2} (p - 2)} \lvert u \rvert^{2^*} \; \d x \bigg)^{\frac{1}{2^*}} 
\leq C \bigg\{ \bigg( \frac{1}{r^d} \int_{\O_{\delta , L} \cap B(x_0 , 2 r)} \lvert u \rvert^p \; \d x \bigg)^{\frac{1}{2}} \\
 &\qquad+ \bigg( \frac{1}{r^d} \int_{ \{ \lvert u \rvert \leq \delta \} \cap B(x_0 , 2 r)} \delta^{p - 2} \lvert u \rvert^2 \; \mathrm{d} x \bigg)^{\frac{1}{2}} + \bigg( \frac{1}{r^d} \int_{ \{ \lvert u \rvert \geq L \} \cap B(x_0 , 2 r)} L^{p - 2} \lvert u \rvert^2 \; \mathrm{d} x \bigg)^{\frac{1}{2}} \bigg\},
\end{aligned}
\end{align}
for some constant $C > 0$ that only depends on $d$, $p$, $M$, $c_{\bullet}$, $c^{\bullet}$ and $\gamma_0$.
Now assume for a moment that $u \in \L^p (\O (x_0 , 2 r))$ and use Fatou's lemma to deduce from~\eqref{Eq: Almost the final inequality}
\begin{align*}
\bigg( \frac{1}{r^d} \int_{\O (x_0 , r / \alpha)} \lvert u \rvert^{\frac{2^* p}{2}}\; \d x \bigg)^{\frac{1}{2^*}} &\leq \liminf_{\substack{\delta \to 0 \\ L \to \infty}} \bigg( \frac{1}{r^d} \int_{\O (x_0 , r / \alpha)} \lvert u \rvert_{\delta , L}^{\frac{2^*}{2} (p - 2)} \lvert u \rvert^{2^*} \; \d x \bigg)^{\frac{1}{2^*}} \\
 &\leq C \bigg( \frac{1}{r^d} \int_{\O (x_0 , 2 r)} \lvert u \rvert^p \; \d x \bigg)^{\frac{1}{2}}.
\end{align*}
For further reference, we rephrase this fact with a slight modification of the notation.
We proved that 
\begin{equation} \label{Eq: Mosers estimate}
\bigg( \frac{1}{r^d} \int_{\O (x_0 , R / \alpha)} \lvert u \rvert^{\frac{2^* q}{2}}\; \d x \bigg)^{\frac{2}{2^* q}} 
 \leq C \bigg( \int_{\O (x_0 , 2 R)} \lvert u \rvert^q \; \mathrm{d} x \bigg)^{\frac{1}{q}}
\end{equation}
for all $x_0 \in \overline \Omega$,
$R > 0$ and $q \in [2 , p_0(\mu))$
with $\Delta_q(\mu) \geq \gamma_0$, if $u \in \L^q(\Omega (x_0 , 2 R))$. 
Here $\alpha = M^2 \sqrt{d}$ and $C > 0$ only depends on $d$, $q$, $M$, $c_{\bullet}$, $c^{\bullet}$ and $\gamma_0$.

Now we iterate~\eqref{Eq: Mosers estimate} to obtain the desired weak reverse H\"older estimate. For this purpose, let $x_0 \in \overline \Omega$,
$r \in (0 , \frac{M}{4})$ and $p \in [2 , p_0(\mu))$ with $\Delta_p(\mu) \geq \gamma_0$.
Let $n_0 = \min \{ k \in \IN : p \leq 2 \cdot \big( \frac{2^*}{2} \big)^k \} $
and define $p_0,\ldots,p_{n_0 + 1} \in [2 , \infty)$ by 
\begin{align*}
\left\{ \begin{aligned} 
p_k &\coloneqq 2 \cdot \big( \tfrac{2^*}{2} \big)^k, \qquad \qquad \text{if } k \in \{ 0,\ldots,n_0-1 \} ,  \\
p_{n_0} &= p ,  \\
p_{n_0 + 1} &= \tfrac{2^*}{2} p .
\end{aligned}
\right.
\end{align*}
Since $[2 , \infty) \ni q \mapsto \Delta_q (\mu)$ is decreasing
by~\cite[Cor.~5.16]{Carbonaro_Dragicevic} it follows that 
$\Delta_{p_k}(\mu) \geq \Delta_p(\mu) \geq \gamma_0$ and in particular $\mu$
is $p_k$-elliptic for all $k \in \{ 0,\ldots,n_0 \}$. Since
$u \in \W^{1 , 2}_D(\O) \subset \L^{p_0} (\O)$, we deduce from~\eqref{Eq: Mosers estimate}
with $q = p_0$ and $R = r$ that
\begin{align*}
 \bigg( \frac{1}{r^d} \int_{\O (x_0 , r / \alpha)} \lvert u \rvert^{p_1} \; \mathrm{d} x \bigg)^{\frac{1}{p_1}} &\leq C \bigg( \int_{\O (x_0 , 2 r)} \lvert u \rvert^2 \; \mathrm{d} x \bigg)^{\frac{1}{2}}.
\end{align*}
It follows that $u \in \L^{p_1} (\O (x_0 , r / \alpha))$. 

Next we perform a second iteration step. Another application of~\eqref{Eq: Mosers estimate} applied in the situation where $q = p_1$ and $R = (2 \alpha)^{-1} r$ gives
\begin{align*}
 \bigg( \frac{1}{r^d} \int_{\O (x_0 , (2 \alpha)^{-1} r / \alpha)} \lvert u \rvert^{p_2} \; \mathrm{d} x \bigg)^{\frac{1}{p_2}} \leq C \bigg( \frac{1}{r^d} \int_{\O (x_0 , r / \alpha)} \lvert u \rvert^{p_1} \; \mathrm{d} x \bigg)^{\frac{1}{p_1}}.
\end{align*}
Applying the zeroth iteration step, it follows
\begin{align*}
 \bigg( \frac{1}{r^d} &\int_{\O (x_0 , (2 \alpha)^{-1} r / \alpha)} \lvert u \rvert^{p_2} \; \mathrm{d} x \bigg)^{\frac{1}{p_2}} \leq C \bigg( \frac{1}{r^d} \int_{\O (x_0 , 2 r)} \lvert u \rvert^2 \; \mathrm{d} x \bigg)^{\frac{1}{2}},
\end{align*}
where $C$ is the product of the constants in the respective inequalities. In particular, it follows that $u \in \L^{p_2} (\O (x_0 , \alpha^{-2} 2^{-1} r))$. Continuing this procedure iteratively yields
\begin{align*}
 \bigg( \frac{1}{r^d} &\int_{\O (x_0 , (2 \alpha)^{- n_0} r / \alpha)} \lvert u \rvert^{p_{n_0 + 1}} \; \mathrm{d} x \bigg)^{\frac{1}{p_{n_0 + 1}}} 
\leq C' \bigg( \frac{1}{r^d} \int_{\O (x_0 , 2 r)} \lvert u \rvert^2 \; \mathrm{d} x \bigg)^{\frac{1}{2}}.
\end{align*}
Altogether, this proves the validity of the desired weak reverse H\"older estimates.
\end{proof}

Following a classical idea, which can be found for example in the book
of Giaquinta~\cite[p.~119]{Giaquinta}, we establish a weak reverse
H\"older estimate for $\nabla u$. Its proof essentially relies on a subsequent
application of Caccioppoli's inequality, Sobolev's embedding and
Poincar\'e's inequality. We start with an adapted version of
Caccioppoli's inequality for mixed boundary conditions and the resolvent
problem.

\begin{lemma}
\label{Lem: Caccioppoli}
Let $\Omega \subset \IR^d$ be a bounded open set, $D \subset \partial \Omega$ be closed, and let
\begin{align*}
 0 < \theta < \pi - \arctan\Big( \frac{2 c^{\bullet}}{c_{\bullet}} \Big) \eqqcolon \pi - \theta_0.
\end{align*}
Then there exists a constant $C_{\theta_0 , \theta} > 0$ such that for all $\lambda \in \Sec_{\theta}$, 
$x_0 \in \overline{\Omega}$, $r > 0$ and $u \in \W^{1 , 2}_D (\O)$ that satisfy
\begin{align*}
 \lambda \int_{\O} u \overline{v} \; \mathrm{d} x + \mathfrak{t} [u , v] = 0
\end{align*}
for all $v \in \W^{1 , 2}_D (\O)$ with $\supp(v) \subset \overline{B(x_0 , 2 r)}$ 
the following is valid.
\begin{tabel}
 \item \label{Lem: Caccioppoli-1}
If $B(x_0 , 2 r) \cap D = \emptyset$, then for all $c \in \IC$
\begin{align*}
  \lvert \lambda \rvert &\int_{\O (x_0 , r)} \lvert u \rvert^2 \; \d x + \frac{c_{\bullet}}{2} \int_{\O (x_0 , r)} \lvert \nabla u \rvert^2 \; \d x \\
 &\leq C_{\theta_0 , \theta} \bigg\{ \lvert \lambda \rvert \lvert c \rvert \int_{\O (x_0 , 2 r)} \lvert u \rvert \; \d x + \frac{2 d^2 C_{\theta_0 , \theta} C_d^2 c^{\bullet 2}}{c_{\bullet}} \frac{1}{r^2} \int_{\O (x_0 , 2 r)} \lvert u + c \rvert^2 \; \d x\bigg\}.
\end{align*}
 \item \label{Lem: Caccioppoli-2}
If $B(x_0 , 2 r) \cap D \neq \emptyset$, then
\begin{align*}
 \lvert \lambda \rvert \int_{\O (x_0 , r)} \lvert u \rvert^2 \; \d x + \frac{c_{\bullet}}{2} \int_{\O (x_0 , r)} \lvert \nabla u \rvert^2 \; \d x \leq \frac{2 d^2 C_{\theta_0 , \theta}^2 C_d^2 c^{\bullet 2}}{c_{\bullet}} \frac{1}{r^2} \int_{\O (x_0 , 2 r)} \lvert u \rvert^2 \; \d x.
\end{align*}
\end{tabel}
\end{lemma}

\begin{proof}
`\ref{Lem: Caccioppoli-1}'.
Since $\theta + \theta_0 < \pi$ there exists a $C_{\theta_0 , \theta} > 0$ such that 
$|z_1| + |z_2| \leq C_{\theta_0 , \theta} |z_1 + z_2|$ for all 
$z_1 \in \overline{ \Sec_{\theta} } $ and $z_2 \in \overline{ \Sec_{\theta_0} } $.
Let $\eta \in \C_c^{\infty} (B(x_0 , 2 r))$ be such that $0 \leq \eta \leq 1$,
$\eta|_{B(x_0 , r)} = \one$ and $\| \nabla \eta \|_{\L^{\infty}} \leq 2 / r$, cf.\@ the proof of Theorem~\ref{Thm: Reverse Holder}.
Let $c \in \IC$ and define $v \coloneqq \eta^2 (u + c) $.
Then $v \in \W^{1 , 2}_D (\O)$ since $B(x_0 , 2 r) \cap D = \emptyset$.
Hence
\begin{align*}
 0 &= \lambda \int_{\Omega} \eta^2 u \overline{(u + c)} \; \d x 
+ \int_{\Omega} \langle \mu \nabla u , \nabla \big( \eta^2 (u + c) \big) \rangle \; \d x \\
 &= \lambda \int_{\Omega} \eta^2 \lvert u \rvert^2 \; \d x 
+ \lambda \overline{c} \int_{\Omega} \eta^2 u \; \d x 
+ \int_{\Omega} \langle \eta^2 \mu \nabla u , \nabla u \rangle \; \d x 
+ 2 \int_{\Omega} \eta \langle \mu \nabla u , (u + c) \nabla \eta \rangle \; \d x.
\end{align*}
The ellipticity of $\mu$ implies that
\begin{align*}
 \int_{\Omega} \eta^2 \langle \mu \nabla u , \nabla u \rangle \; \d x  
\in \overline{ \Sec_{\theta_0}}.
\end{align*}
Therefore
\[ \lvert \lambda \rvert \int_{\O} \eta^2 \lvert u \rvert^2 \; \d x 
+ c_{\bullet} \int_{\O} \eta^2 \lvert \nabla u \rvert^2 \; \d x
 \leq C_{\theta_0 , \theta} \bigg\{ \lvert \lambda \rvert \lvert c \rvert \int_{\O} \eta^2 \lvert u \rvert \; \d x 
+ 2 d c^{\bullet} \int_{\O} \eta \lvert \nabla u \rvert \lvert \nabla \eta \rvert \lvert u + c \rvert \; \d x \bigg\}.
\]
Consequently, by Young's inequality
\begin{align*}
 2 d C_{\theta_0 , \theta} c^{\bullet} \int_{\O} \eta \lvert \nabla u \rvert \lvert \nabla \eta \rvert \lvert u + c \rvert \; \d x 
\leq \frac{c_{\bullet}}{2} \int_{\O} \eta^2 \lvert \nabla u \rvert^2 \; \d x + \frac{2 d^2 C_{\theta_0 , \theta}^2 c^{\bullet 2}}{c_{\bullet}} \int_{\O} \lvert \nabla \eta \rvert^2 \lvert u + c \rvert^2 \; \d x
\end{align*}
and the term involving the gradient of $u$ can be absorbed into the left-hand side. The properties of $\eta$ then yield
\begin{align*}
 \lvert \lambda \rvert \int_{\O (x_0 , r)} &\lvert u \rvert^2 \; \d x + \frac{c_{\bullet}}{2} \int_{\O (x_0 , r)} \lvert \nabla u \rvert^2 \; \d x \\
 &\leq C_{\theta_0 , \theta} \bigg\{ \lvert \lambda \rvert \lvert c \rvert \int_{\O (x_0 , 2 r)} \lvert u \rvert \; \d x + \frac{2 d^2 C_{\theta_0 , \theta} C_d^2 c^{\bullet 2}}{c_{\bullet}} \frac{1}{r^2} \int_{\O (x_0 , 2 r)} \lvert u + c \rvert^2 \; \d x\bigg\}
\end{align*}
as required.

`\ref{Lem: Caccioppoli-2}'.
The argument is almost as above, but now one has to choose $c = 0$, i.e., $v = \eta^2 u$.
\end{proof}

We continue by combining Lemma~\ref{Lem: Caccioppoli} with the local Sobolev embedding (Lemma~\ref{Lem: Local Sobolev embedding}) and the local Poincar\'e inequality (Lemma~\ref{Lem: Local Poincare}).

\begin{lemma}
\label{Lem: Higher integrability of gradient}
Let $\Omega \subset \IR^d$ be open and bounded, where $d \geq 3$.
Let $D \subset \partial \Omega$ be closed and set $N = \partial \Omega \setminus D$.
Adopt Assumptions~N,~$\Omega^c$ and~D, 
and let $\mu \in \L^{\infty} (\Omega ; \IC^{d \times d})$ be 
subject to Assumption~\ref{Ass: Ellipticity}. 
Let $0 < \theta < \pi$ be chosen as in Lemma~\ref{Lem: Caccioppoli}. 
Then there exists a
constant $C > 0$ such that for all $\lambda \in \Sec_{\theta}$,
 $x_0 \in \overline{\Omega}$,
$0 < r < \min\{ s_0 / 4 , r_0 / (4 \iota) , r_0 / 2 , 1 / (16 M \sqrt{d}) \}$ and 
$u \in \W^{1 , 2}_D (\O)$ satisfying
\begin{align*}
 \lambda \int_{\O} u \overline{v} \; \mathrm{d} x + \mathfrak{t} [u , v] = 0
\end{align*}
for all $v \in \W^{1 , 2}_D (\O)$ with $\supp(v) \subset \overline{B(x_0 , 2 r)}$ 
it follows that 
\begin{align*}
 \bigg( \frac{1}{r^d} \int_{\O (x_0 , r)} \big\{ \lvert \lambda u \rvert &+ \lvert \lambda \rvert^{\frac{1}{2}} \lvert \nabla u \rvert \big\}^2 \; \d x \bigg)^{\frac{1}{2}} \leq C \bigg( \frac{1}{r^d} \int_{\O (x_0 , 8 M^2 \sqrt{d} r)} \big\{ \lvert \lambda u \rvert + \lvert \lambda \rvert^{\frac{1}{2}} \lvert \nabla u \rvert \big\}^{2_*} \; \d x \bigg)^{\frac{1}{2_*}},
\end{align*}
where $2_* \coloneqq \frac{2d}{d + 2}$.
\end{lemma}

\begin{proof}
If $B(x_0 , 2 r) \subset \O$, then Lemma~\ref{Lem: Caccioppoli} with 
$c \coloneqq - u_{B(x_0 , 2 r)}$ gives
\begin{align*}
 \lvert \lambda \rvert \int_{B(x_0 , r)} &\lvert u \rvert^2 \; \d x + \int_{B(x_0 , r)} \lvert \nabla u \rvert^2 \; \d x \\
 &\leq C \bigg\{ \lvert \lambda \rvert \lvert u_{B(x_0 , 2 r)} \rvert \int_{B(x_0 , 2 r)} \lvert u \rvert \; \d x + \frac{1}{r^2} \int_{B(x_0 , 2 r)} \lvert u - u_{B(x_0 , 2 r)} \rvert^2 \; \d x\bigg\}
\end{align*}
for a suitable constant $C > 0$. The first term on the right-hand side is controlled by 
H\"older's inequality and the second term is controlled by means of a Sobolev--Poincar\'e inequality. 
This altogether yields
\begin{align*}
 \lvert \lambda \rvert \int_{B(x_0 , r)} &\lvert u \rvert^2 \; \d x + \int_{B(x_0 , r)} \lvert \nabla u \rvert^2 \; \d x \\
  &\leq C' r^d \bigg\{\lvert \lambda \rvert \bigg( \frac{1}{r^d} \int_{B(x_0 , 2 r)} \lvert u \rvert^{2_*} \; \d x \bigg)^{\frac{2}{2_*}} + \bigg( \frac{1}{r^d} \int_{B(x_0 , 2 r)} \lvert \nabla u \rvert^{2_*} \; \d x \bigg)^{\frac{2}{2_*}}\bigg\}.
\end{align*}
Now, multiplying by $\lvert \lambda \rvert$, dividing by $r^d$, 
and taking the square root of the inequality leads us to
\begin{align*}
 \bigg( \frac{1}{r^d} \int_{B(x_0 , r)} &\lvert \lambda u \rvert^2 + \big( \lvert \lambda \rvert^{\frac{1}{2}} \lvert \nabla u \rvert \big)^2 \; \d x \bigg)^{\frac{1}{2}} \\
  &\leq C \bigg\{ \bigg( \frac{1}{r^d} \int_{B(x_0 , 2 r)} \lvert \lambda u \rvert^{2_*} \; \d x \bigg)^{\frac{2}{2_*}} + \bigg( \frac{1}{r^d} \int_{B(x_0 , 2 r)} \big( \lvert \lambda \rvert^{\frac{1}{2}} \lvert \nabla u \rvert \big)^{2_*} \; \d x \bigg)^{\frac{2}{2_*}}\bigg\}^{\frac{1}{2}}.
\end{align*}
Finally, the equivalence of norms in $\IR^2$, in particular the estimates
\begin{align*}
 (a + b)^2 \leq 2(a^2 + b^2), \quad a^{\frac{2}{2_*}} + b^{\frac{2}{2_*}} \leq (a + b)^{\frac{2}{2_*}}
 \quad \text{and} \quad a^{2_*} + b^{2_*} \leq (a + b)^{2_*}
\end{align*}
for all $a,b \in [0,\infty)$ yield the desired estimate for interior balls. 

If $x_0 \in \partial \Omega$ and $B(x_0 , 2 r) \cap D = \emptyset$, then
we do almost the same, but we take $c \coloneqq - (E u)_{B(x_0 , 2 r)}$,
where $E$ denotes the local reflection operator from Remark~\ref{Rem: Reflection at Lipschitz boundary}.
In this case, Lemma~\ref{Lem: Caccioppoli} gives
\begin{align*}
 \lvert \lambda \rvert &\int_{\O (x_0 , r)} \lvert u \rvert^2 \; \d x + \int_{\O (x_0 , r)} \lvert \nabla u \rvert^2 \; \d x \\
 &\leq C \bigg\{ \lvert \lambda \rvert \lvert (Eu)_{B(x_0 , 2 r)} \rvert \int_{\O (x_0 , 2 r)} \lvert u \rvert \; \d x + \frac{1}{r^2} \int_{B(x_0 , 2 r)} \lvert E u - (E u)_{B(x_0 , 2 r)} \rvert^2 \; \d x\bigg\} \\
 &\leq C r^d \bigg\{ \lvert \lambda \rvert \lvert (Eu)_{B(x_0 , 2 r)} \rvert \frac{1}{r^d} \int_{\O (x_0 , 2 r)} \lvert u \rvert \; \d x + \bigg( \frac{1}{r^d} \int_{B(x_0 , 2 r)} \lvert \nabla (E u) \rvert^{2_*} \; \d x \bigg)^{\frac{2}{2_*}} \bigg\}.
\end{align*}
The estimates given in Remark~\ref{Rem: Reflection at Lipschitz boundary} imply
\begin{align*}
 \bigg( \frac{1}{r^d} \int_{B(x_0 , 2 r)} \lvert \nabla (E u) \rvert^{2_*} \; \d x \bigg)^{\frac{2}{2_*}} 
\leq C' \bigg( \frac{1}{r^d} \int_{\O (x_0 , 2 M^2 \sqrt{d} r)} \lvert \nabla u \rvert^{2_*} \; \d x \bigg)^{\frac{2}{2_*}},
\end{align*}
where $C' > 0$ only depends on $d$ and $M$ and
\begin{align*}
 \lvert (Eu)_{B(x_0 , 2 r)} \rvert \leq \frac{C''}{r^d} \int_{\O (x_0 , 2 M^2 \sqrt{d} r)} \lvert u \rvert \; \d x.
\end{align*}
Using once again the H\"older inequality we arrive at
\begin{align*}
 \lvert \lambda \rvert &\int_{\O (x_0 , r)} \lvert u \rvert^2 \; \d x + \int_{\O (x_0 , r)} \lvert \nabla u \rvert^2 \; \d x \\
 &\leq C''' r^d \bigg\{ \lvert \lambda \rvert \bigg( \frac{1}{r^d} \int_{\O (x_0 , 2 M^2 \sqrt{d} r)} \lvert u \rvert^{2_*} \; \d x \bigg)^{\frac{2}{2_*}} + \bigg( \frac{1}{r^d} \int_{\O (x_0 , 2 M^2 \sqrt{d} r)} \lvert \nabla u \rvert^{2_*} \; \d x \bigg)^{\frac{2}{2_*}} \bigg\}.
\end{align*}
Multiplying again by $\lvert \lambda \rvert$,
dividing by $r^d$, taking the square root and employing the equivalence
of norms in $\IR^2$ as above establishes the desired inequality.

If $x_0 \in \partial \Omega$ and $B(x_0 , 2 r) \cap D \neq \emptyset$, then Lemma~\ref{Lem: Caccioppoli} implies
\begin{align*}
 \lvert \lambda \rvert \int_{\O (x_0 , r)} \lvert u \rvert^2 \; \d x + \int_{\O (x_0 , r)} \lvert \nabla u \rvert^2 \; \d x \leq  \frac{C}{r^2} \int_{\O (x_0 , 2 r)} \lvert u \rvert^2 \; \d x.
\end{align*}
Now Lemma~\ref{Lem: Local Sobolev embedding} guarantees that
\[ \frac{1}{r^2} \int_{\O (x_0 , 2 r)} \lvert u \rvert^2 \; \d x
 \leq C r^d \bigg\{ \bigg( \frac{1}{r^d} \int_{\O (x_0 , 2 M^2 \sqrt{d} r)} \lvert \nabla u \rvert^{2_*} \; \d x \bigg)^{\frac{1}{2_*}} 
     + \frac{1}{r} \bigg( \frac{1}{r^d} \int_{\O (x_0 , 2 r)} \lvert u \rvert^{2_*} \; \d x \bigg)^{\frac{1}{2_*}} \bigg\}^2.
\]
An application of Lemma~\ref{Lem: Local Poincare} with $p = 2_*$  
yields the desired inequality.
\end{proof}

\section{From weak reverse H\"older estimates to $\L^p$-estimates}
\label{Sec: From weak reverse Holder estimates to operator bounds}

In this section we provide the proofs of Theorems~\ref{Thm: Analyticity} 
and~\ref{Thm: Gradient estimates for resolvent}. 
The proofs fundamentally base on the following $\L^p$-extrapolation theorem of 
Shen~\cite[Thm.~3.3]{Shen} which was initially proved on bounded Lipschitz domains 
and generalised in~\cite[Thm.~4.1]{Tolksdorf} to general bounded measurable sets.

\begin{theorem}
\label{Thm: Lp-extrapolation theorem}
Let $\Omega \subset \IR^d$ be bounded and Lebesgue-measurable, let $k , m \in \IN$,
$\mathcal{M} > 0$, and let $T \in \Lop(\L^2 (\Omega ; \IC^k) , \L^2 (\Omega ; \IC^m))$ with 
$\| T \|_{\Lop(\L^2 (\Omega ; \IC^k) , \L^2 (\Omega ; \IC^m))} \leq \mathcal{M}$.
Further let $p > 2$, $R_0 > 0$, $\alpha_2 > \alpha_1 > 1$
and $\mathcal{C} > 0$.
Suppose that 
\[
 \bigg( \frac{1}{r^d} \int_{\Omega (x_0 , r)} \lvert T f \rvert^p \; \d x \bigg)^{1 / p} 
\leq \mathcal{C} \bigg( \frac{1}{r^d} \int_{\Omega (x_0 , \alpha_1 r)} \lvert T f \rvert^2 \; \d x \bigg)^{1 / 2}
\]
 for all $0 < r < R_0$ and $x_0$ that either satisfy $x_0 \in \partial \Omega$ or 
$B(x_0 , \alpha_2 r) \subset \Omega$ and for all compactly supported $f \in \L^{\infty}(\Omega ; \IC^k)$ with 
$f = 0$ on $\Omega(x_0 , \alpha_2 r)$.

Then $T$ restricts for all $2 < q < p$ to a bounded operator in 
$\Lop(\L^q (\Omega ; \IC^k) , \L^q (\Omega ; \IC^m))$ with an operator norm that is bounded by a constant depending on 
$d$, $p$, $q$, $\alpha_1$, $\alpha_2$, $\mathcal{C}$, $\mathcal{M}$, $R_0$ and $\diam(\Omega)$.
\end{theorem}

Now we are in the position to give a proof of Theorem~\ref{Thm: Analyticity}.

\begin{proof}[Proof of Theorem~\ref{Thm: Analyticity} and Theorem~\ref{Thm: Gradient estimates for resolvent}\ref{Thm: Gradient estimates for resolvent-1}]
By the Lax--Milgram lemma, there are $\theta_1 \in (\pi / 2 , \pi)$ and $\mathcal{M} > 0$, 
depending only on $c_\bullet$ and $c^\bullet$ such that 
$\Sec_{\theta_1} \subset \rho(- A_2)$ and 
\[
  \| \lambda (\lambda+ A_2)^{-1} f \|_{\L^2(\Omega)} 
\leq \mathcal{M} \| f \|_{\L^2(\Omega)} \qquad (f \in \L^2(\Omega) , \; \lambda \in \Sec_{\theta_1}).
\]
Similarly, it follows that
\begin{align}
\label{Eq: Resolvent divergence estimate}
 \lvert \lambda \rvert^{\frac{1}{2}} \| (\lambda+ A_2)^{-1} \dive g \|_{\L^2(\Omega)} 
\leq \mathcal{M} \| g \|_{\L^2(\Omega ; \IC^d)} 
   \qquad (g \in \L^2(\Omega ; \IC^d) , \; \lambda \in \Sec_{\theta_1}).
\end{align}
Here the operator $(\lambda + A_2)^{-1} \dive$ is understood as the solution operator to the equation
\begin{align*}
 \lambda \int_{\O} u \overline{v} \; \d x + \mathfrak{t} [u , v] = - \int_{\O} \langle g , \nabla v \rangle \; \d x \qquad (v \in \W^{1 , 2}_D (\O)).
\end{align*}
Let $2 < p < p_0 (\mu)$. Furthermore, let $\theta_2 \in (\pi / 2 , \pi)$ denote the angle
obtained in Theorem~\ref{Thm: Reverse Holder}. 
Define $\theta \coloneqq \min\{ \theta_1 , \theta_2 \}$. 
Let $0 < c < 1$ be the constant (which only depends on $d$, $p$ and
$M$) as in Theorem~\ref{Thm: Reverse Holder}.
Let $\mathcal{C}$ be as in Theorem~\ref{Thm: Reverse Holder}.
Note that $\mathcal{C}$ depends only on $d$, $p$, $M$, $c_{\bullet}$, $c^{\bullet}$ and 
$\gamma_0$.
For all $\lambda \in \Sec_{\theta}$ define
\begin{align*}
 T_{\lambda} \colon \L^2 (\Omega ; \IC^{1 + d}) \to \L^2(\Omega ; \IC^{1 + 1}), 
\quad T_\lambda(f , g) = \begin{pmatrix} \lambda (\lambda + A_2)^{-1} f \\ \lvert \lambda \rvert^{\frac{1}{2}} (\lambda+ A_2)^{-1} \dive g \end{pmatrix}.
\end{align*}
Let $0 < r < c M / 4 \eqqcolon R_0$ and $x_0 \in \overline{\O}$.
Suppose either $x_0 \in \partial \O$ or $B(x_0 , (3 / c) r) \subset \O$.
Let $(f , g) \in \L^{\infty} (\O ; \IC^{1 + d})$ and suppose that $(f , g)$ vanishes on 
$\O (x_0 , (3 / c) r)$. 
Then 
\begin{align*}
 \bigg( \frac{1}{r^d} \int_{\O (x_0 , r)} \lvert T_{\lambda} (f , g) \rvert^{\frac{p d}{d - 2}} \; \d x \bigg)^{\frac{d - 2}{p d}} \leq \mathcal{C} \bigg( \frac{1}{r^d} \int_{\O (x_0 , (2 / c) r)} \lvert T_{\lambda} (f , g) \rvert^2 \; \d x \bigg)^{\frac{1}{2}},
\end{align*}
by Theorem~\ref{Thm: Reverse Holder}.
Taking $\alpha_1 \coloneqq 2 / c$ and $\alpha_2 \coloneqq 3 / c$ it follows from Theorem~\ref{Thm: Lp-extrapolation theorem} that $T_{\lambda}$ restricts to a bounded operator on $\L^q (\Omega)$ for all 
$q \in (2 , p d / (d - 2))$. 
Moreover, the operator norm is bounded by a
constant that depends only on $d$, $p$, $q$, $M$, $c$, $\mathcal{C}$, $\mathcal{M}$, 
$\gamma_0$ and $\diam(\O)$.
This implies that there exists a constant $C > 0$ such that
\begin{align*}
 \lvert \lambda \rvert \| (\lambda + A_2)^{-1} f \|_{\L^q (\Omega)} \leq C \| f \|_{\L^q (\Omega)} 
\qquad (f \in \L^q (\Omega) ,\;  \lambda \in \Sec_{\theta})
\end{align*}
and
\begin{align*}
 \lvert \lambda \rvert^{\frac{1}{2}} \| (\lambda + A_2)^{-1} \dive g \|_{\L^q (\Omega)} 
\leq C \| g \|_{\L^q (\Omega ; \IC^d)} 
\qquad (g \in \L^q (\Omega ; \IC^d) , \; \lambda \in \Sec_{\theta}).
\end{align*}
Since $A_q$ is the part of $A_2$ in $\L^q(\Omega)$ this immediately yields
$\Sec_{\theta} \subset \rho(- A_q)$ and 
$(\lambda + A_2)^{-1} f = (\lambda + A_q)^{-1} f$ for all $f \in \L^q (\Omega)$ 
and $\lambda \in \Sec_{\theta}$. 
Consequently, $A_q$ is a sectorial operator in $\L^q (\O)$ and by~\cite[Prop.~2.1.1(h)]{Haase} 
it is densely defined. 
Now Lemma~\ref{Lem: Duality} yields that $(A_2^*)_{q^{\prime}}$ is
well-defined and by duality it is sectorial  of the same angle. 
The operator $A_2^*$ is
the divergence form operator that is associated to the matrix $\mu^*$.
Finally, since $(\mu^*)^* = \mu$ and because $p_0(\mu) = p_0 (\mu^*)$,
by~\cite[Cor.~5.17]{Carbonaro_Dragicevic}, we find by the argumentation
above with $\mu$ replaced by $\mu^*$ that $A_{q^{\prime}}$ is
well-defined and sectorial. Moreover, the gradient estimate in Theorem~\ref{Thm: Gradient estimates for resolvent}\ref{Thm: Gradient estimates for resolvent-1} follows by duality as well.
Since $2 < p < p_0 (\mu)$ is
arbitrary this concludes the proof.
\end{proof}

The proof of Theorem~\ref{Thm: Gradient estimates for resolvent}\ref{Thm: Gradient estimates for resolvent-2} and~\ref{Thm: Gradient estimates for resolvent-3} is similar. 

\begin{proof}[Proof of Theorem~\ref{Thm: Gradient estimates for resolvent}\ref{Thm: Gradient estimates for resolvent-2} and~\ref{Thm: Gradient estimates for resolvent-3}]

Dualising~\eqref{Eq: Resolvent divergence estimate} implies the
estimate
\begin{align*}
 \lvert \lambda \rvert^{\frac{1}{2}} \| \nabla (\lambda+ A_2)^{-1} f \|_{\L^2(\Omega ; \IC^d)} \leq \mathcal{M} \| f \|_{\L^2(\Omega)} \qquad (f \in \L^2(\Omega) , \lambda \in \Sec_{\theta_1}).
\end{align*}
Moreover, the Lax--Milgram lemma implies the validity of
\begin{align*}
 \| \nabla (\lambda + A_2)^{-1} \dive g \|_{\L^2 (\Omega ; \IC^d)} \leq \mathcal{M} \| g \|_{\L^2 (\Omega ; \IC^d)} \qquad (g \in \L^2 (\Omega ; \IC^d) , \lambda \in \Sec_{\theta_1}).
\end{align*}
Let $\theta_2 \in (\pi / 2 , \pi)$ denote the angle as in Lemma~\ref{Lem: Caccioppoli} 
and define $\theta \coloneqq \min\{ \theta_1 , \theta_2 \}$. 
Define the operator
\begin{align*}
 S_{\lambda} \colon \L^2 (\O ; \IC^{1 + d}) \to \L^2 (\O ; \IC^{1 + d + 1 + d}), 
\quad S_\lambda(f , g) = \begin{pmatrix} \lambda (\lambda + A_2)^{-1} f \\ \lvert \lambda \rvert^{\frac{1}{2}} \nabla (\lambda + A_2)^{-1} f \\ \lvert \lambda \rvert^{\frac{1}{2}} (\lambda + A_2)^{-1} \dive g  \\ \nabla (\lambda + A_2)^{-1} \dive g \end{pmatrix},
\end{align*}
which defines a uniformly bounded family of operators on $\L^2$. 

Let $0 < r < \min\{ s_0 / 4 , r_0 / (4 \iota) , r_0 , 1 / (16 M \sqrt{d}) \}$ 
and let $x_0 \in \overline{\Omega}$ satisfy either $x_0 \in \partial \Omega$ 
or $B(x_0 , 16 M^2 \sqrt{d} r) \subset \O$. Let $(f , g) \in \L^{\infty} (\O ; \IC^{1 + d})$ 
vanish on $\O (x_0 , 16 M^2 \sqrt{d} r)$. Then by
Lemma~\ref{Lem: Higher integrability of gradient} there exists a
constant $C > 0$ such that the weak reverse H\"older estimate
\begin{align*}
 \bigg( \frac{1}{r^d} \int_{\O (x_0 , r)} \lvert S_{\lambda} (f , g) \rvert^2 \; \d x \bigg)^{\frac{1}{2}} \leq C \bigg( \frac{1}{r^d} \int_{\O (x_0 , 8 M^2 \sqrt{d} r)} \lvert S_{\lambda} (f , g) \rvert^{\frac{2 d}{d + 2}} \; \d x \bigg)^{\frac{d + 2}{2 d}}
\end{align*}
is valid. 
By~\cite[Lem.~4.2]{Tolksdorf} the weak reverse H\"older estimate is even valid for all 
$x_0 \in \IR^d$ with a different constant $C$ (independent of $x_0$ and $r$). 
Consequently, the self-improving property of weak reverse H\"older estimates, 
see e.g.,~\cite[Thm.~6.38]{Giaquinta_Martinazzi}, establishes the existence of an $\eps > 0$
 such that the weak reverse H\"older estimate
\begin{align*}
 \bigg( \frac{1}{r^d} \int_{\O(x_0 , r)} \lvert S_{\lambda} (f , g) \rvert^p \; \d x \bigg)^{\frac{1}{p}} \leq C \bigg( \frac{1}{r^d} \int_{\O(x_0 , 8 M^2 \sqrt{d} r)} \lvert S_{\lambda} (f , g) \rvert^2 \; \d x \bigg)^{\frac{1}{2}}
\end{align*}
is valid for all $2 < p < 2 + \eps$.
It follows that $S_{\lambda}$ restricts to a uniformly bounded family of operators 
from $\L^q (\Omega ; \IC^{1 + d})$ into $\L^q (\Omega ; \IC^{1 + d + 1 + d})$ for all $2 < q < 2 + \eps$. Noticing that the boundedness of $S_{\lambda}$ implies the boundedness of the operators in each row proves the theorem.
\end{proof}

\section{Perturbation of real-valued matrices}
\label{Sec: Perturbation of real-valued matrices}

This section is devoted to the perturbation theory of elliptic operators with complex coefficients. First, we record the following lemma that shows that $p$-ellipticity is stable under small perturbations of the coefficients.

\begin{lemma} \label{lem: 601}
Let $\Omega \subset \IR^d$ be open and bounded, where $d \geq 3$.
Let $\mu \in \L^{\infty} (\Omega ; \IC^{d \times d})$ be 
subject to Assumption~\ref{Ass: Ellipticity}. 
Let $\nu \in \L^{\infty} (\Omega ; \IC^{d \times d})$ and suppose that 
$\|\nu\|_{\L^{\infty} (\Omega ; \Lop(\IC^d))} < c_{\bullet}$.
Then $\mu + \nu$ is subject to Assumption~\ref{Ass: Ellipticity}.
Moreover, 
\[
\Delta_p(\mu + \nu) 
\geq \Delta_p(\mu) - 2 \max \Big(\frac{1}{p} , \frac{1}{p'} \Big) \|\nu\|_{\L^{\infty} (\Omega ; \Lop(\IC^d))}
\]
for all $p \in (1,\infty)$.
In particular, if $\mu$ is $p$-elliptic and 
$2 \max(\frac{1}{p} , \frac{1}{p'}) \|\nu\|_{\L^{\infty} (\Omega ; \Lop(\IC^d))} < \Delta_p(\mu)$,
then $\mu + \nu$ is $p$-elliptic.
\end{lemma}
\begin{proof}
The first part is easy.

Let $\alpha,\beta \in \IR^d$.
Write $\xi = \alpha + \ii \beta$.
Let $x \in \Omega$.
Then 
$|\cJ_p(\xi)|^2 
= \frac{4}{(p^{\prime})^2} |\alpha|^2 + \frac{4}{p^2} |\beta|^2 
\leq 4 \max(\frac{1}{p^2} , \frac{1}{(p')^2}) |\xi|^2$.
So 
\[
|\langle \nu(x) \xi , \cJ_p(\xi) \rangle|
\leq \|\nu\|_{\L^{\infty} (\Omega ; \Lop(\IC^d))} |\xi| |\cJ_p(\xi)|
\leq 2 \max \Big(\frac{1}{p} , \frac{1}{p'} \Big) \|\nu\|_{\L^{\infty} (\Omega ; \Lop(\IC^d))} |\xi|^2
\]
and the lemma follows.
\end{proof}

Next, we prove Theorem~\ref{tnumrange307}. Recall that Lemma~\ref{Lem: Control of imaginary part} states that
$\mu$ is $p$-elliptic for all $p \in (1 , \infty)$ if and only if $\Im (\mu)$ = 0. Thus, Theorem~\ref{tnumrange307}
\textit{cannot} be concluded by perturbing the $p$-ellipticity as in
Lemma~\ref{lem: 601} and then by applying Theorem~\ref{Thm: Analyticity}.
Instead, the proof is based on the Gaussian kernel estimates obtained in 
\cite[Thm.~7.5]{ERe2}, which were a consequence of De Giorgi estimates
for the operator $A_2$.
The De Giorgi estimates allow to have a (complex) perturbation. 
This was shown first for operators on $\IR^d$ by Auscher~\cite[Thm.~4.4]{Aus1}.
We modify the proofs in~\cite{Aus1} and~\cite{ERe2} for our situation.
In the proof we need the next well-known lemma of Campanato, which allows for the perturbation of the De Giorgi
estimates, cf.\@ \cite[Lemma~5.12]{Giaquinta_Martinazzi} or \cite[Lemma~III.2.1]{Giaquinta}.

\begin{lemma} \label{ldir921}
For all $c,\alpha,\beta > 0$ with $\alpha > \beta$ there exist
$\varepsilon > 0$ and $\tilde c > 0$ {\rm (}depending only on $c$, $\alpha$ and $\beta${\rm )}
such that the following is valid.

Let $B \geq 0$, $R_0 > 0$, and let
$\Psi \colon (0,R_0] \to [0,\infty)$ be an increasing function
with the property that
\[
\Psi(r) \leq c \Big( \Big( \frac{r}{R} \Big)^\alpha + \varepsilon \Big) \Psi(R) + B R^\beta
\]
for all $r,R \in \IR$ with $0 < r \leq R \leq R_0$.
Then
\[
\Psi(r) \leq \tilde c \Big( \Big( \frac{r}{R} \Big)^\beta  \Psi(R) + B r^\beta \Big)
\]
for all $0 < r \leq R \leq R_0$.
\end{lemma}

In the forthcoming proofs it will be important to not emphasise the set $D$, where functions in $\W^{1 , 2}_D (\Omega)$ vanish, but to emphasise the complementary set $N = \partial \Omega \setminus D$. Thus, in order to simplify the notation, we introduce
\begin{align*}
 \widehat{\W}^{1 , 2}_N (\Omega) \coloneqq \W^{1 , 2}_{\partial \Omega \setminus N} (\Omega).
\end{align*}
Moreover, for all $x \in \IR^d$ and $r > 0$ we shall write in the following $N(x,r) = N \cap B(x,r)$ and $Q(x,r) = Q \cap B(x,r)$, where $Q$ will be $(- 1 , 1)^{d - 1} \times (0 , 1)$. Recall that we already write $\Omega (x , r)$ for $\Omega \cap B(x , r)$.

\begin{definition}
Let $\Omega \subset \IR^d$ be open, $N$ a relatively open subset 
of $\partial \Omega$ and $\Delta \subset \partial \Omega \setminus N$ be closed.
Let $\mu \in \L^{\infty} (\Omega ; \IC^{d \times d})$ subject to Assumption~\ref{Ass: Ellipticity}
and let $A_2$ be the associated operator.
Let $\kappa_0 \in (0,1)$, $c_{DG} > 0$ and $\Upsilon \subset \overline \Omega$ 
be a set.
Then we say that $A_2$ satisfies $(\kappa_0,c_{DG})$-\emph{De Giorgi
estimates} on $\Upsilon$ for functions vanishing on $\Delta$
and Neumann boundary conditions on $N$
if 
\[
\int_{\Omega (x,r)} |\nabla u|^2 \; \d x
\leq c_{DG} \Big( \frac{r}{R} \Big)^{d-2+2\kappa_0} \int_{\Omega (x,R)} |\nabla u|^2 \; \d x
\]
for all $x \in \Upsilon$, $0 < r \leq R \leq 1$ and $u \in \W^{1,2}_{\Delta}(\Omega)$ satisfying
\[
\int_{\Omega (x,R)}
    \langle \mu \nabla u , \nabla v \rangle  \; \d x
= 0
\]
for all $v \in \widehat \W^{1,2}_{N(x,R)}(\Omega (x,R))$.
\end{definition}

We start with a perturbation of~\cite[Prop.~5.3]{ERe2}. As we work in the following only in the $\L^2$ scale, we drop the index $p = 2$ at the operator. 
Furthermore, we adopt the notation $A^{\mu}$ to indicate that $A^{\mu}$ is the divergence-form operator with coefficients $\mu$.

\begin{proposition} \label{pdir922}
Let $\Omega \subset \IR^d$ be open, $D \subset \partial \Omega$ be closed and
$N := \partial \Omega \setminus D$ subject to Assumption~N. Let $x_0 \in \overline{N} \cap D$
and let $\Phi \coloneqq \Phi_{x_0}$ be the bi-Lipschitz homeomorphism of Assumption~N with corresponding neighbourhood $U \coloneqq U_{x_0}$.
In addition, suppose that Assumption~P is valid and let $\mu \in \L^{\infty} (\Omega ; \IR^{d \times d})$ be subject to Assumption~\ref{Ass: Ellipticity}.

Then there exists an $\varepsilon > 0$ such that for all (complex valued)
$\nu \in \L^{\infty} (\Omega ; \IC^{d \times d})$ satisfying
$\|\nu\|_{\L^{\infty} (\Omega ; \Lop(\IC^d))} < \varepsilon$
the operator $A^{\mu + \nu}$ is m-sectorial.
Moreover, if $A_{\Phi}^{\mu + \nu}$ denotes the operator in $(-1,1)^{d-1} \times (0,1)$ obtained
from $A^{\mu + \nu}$ under the transformation $\Phi$, there are $\kappa_0 \in (0,1)$ and $c_{DG} > 0$ such that the operator
$A^{\mu + \nu}_{\Phi}$ satisfies 
$(\kappa_0,c_{DG})$-De Giorgi estimates on $(-\frac{1}{2},\frac{1}{2})^{d-1} \times (0,\frac{1}{2})$ 
for functions vanishing on $\overline{\Phi(U \cap  D)}$
and Neumann boundary conditions on $\Phi(U \cap N)$.
The constants $\varepsilon$, $\kappa_0$ and $c_{DG}$ can be chosen to depend
only on $d$, $M$, $c_\bullet$, $c^\bullet$, $c_0$ and $c_1$.
\end{proposition}
\begin{proof}
By~\cite[Prop.~5.3]{ERe2} there exist 
$\kappa_0 \in (0,1)$ and $c_{DG} > 0$ such that the transformed operator
$A_\Phi^\mu$ satisfies 
$(\kappa_0,c_{DG})$-De Giorgi estimates on $(-\frac{1}{2},\frac{1}{2})^{d-1} \times (0,\frac{1}{2})$ 
for functions vanishing on $\overline{\Phi(U \cap  D)}$
and Neumann boundary conditions on $\Phi(U \cap N)$.
Set $c = 4 c_{DG}$, $\alpha = d - 2 + 2 \kappa_0$ and $\beta = d - 2 + \kappa_0$.
Let $\varepsilon > 0$ and $\tilde c > 0$ be as in Lemma~\ref{ldir921}.
Define
\[
\tilde \varepsilon
= \frac{c_\bullet}{2} \wedge
     \frac{\sqrt{\varepsilon} c_\bullet}{2 (d! M^{d+2})^2} \Big( \frac{4 c_{DG}}{2 + 4 c_{DG}} \Big)^{1/2}
.  \]
Then 
\[
\frac{2 + 4 c_{DG}}{4 c_{DG}}  (d! M^{d+2})^4 \frac{4 \tilde \varepsilon^2}{c_\bullet^2} \leq \varepsilon
\]
and $\tilde \varepsilon \in (0,c_\bullet)$.

Let $\nu \in \L^{\infty} (\Omega ; \IC^{d \times d})$ with
$\|\nu\|_{\L^{\infty} (\Omega ; \Lop(\IC^d))} < \tilde \varepsilon$.
Let $\mu^\Phi$ and $(\mu + \nu)^\Phi$ be the coefficient functions on 
$(-1,1)^{d-1} \times (0,1)$ obtained
from $\mu$ and $\mu + \nu$ under the transformation $\Phi$.
Define $Q = (-1,1)^{d-1} \times (0,1)$, 
$D_Q = \overline{\Phi(U \cap D)}$ and $N_Q = \Phi(U \cap N)$.
Let $u \in \W^{1,2}_{D_Q}(Q)$ and $x \in \frac{1}{2} Q$.
Fix $R_0 \in (0,1]$.
Suppose that 
\[
\int_{Q (x,R_0)}
    \langle (\mu + \nu)^\Phi \nabla u, \nabla v \rangle \; \d x
= 0
\]
for all $v \in \widehat \W^{1,2}_{N_Q(x,R)}(Q (x,R_0))$.
Let $0 < r \leq R \leq R_0$.
By~\cite[Lem.~6.1(b)]{ERe2} and the Lax--Milgram lemma there exists a unique
$\tilde v \in \widehat \W^{1,2}_{N_Q(x,R)}(Q (x,R))$
such that 
\begin{equation}
\int_{Q (x,R)}
    \langle \mu^\Phi \nabla \tilde v, \nabla \varphi \rangle \; \d x
= \int_{Q (x,R)}
    \langle \mu^\Phi \nabla u, \nabla \varphi \rangle \; \d x
\label{epdir922;4}
\end{equation}
for all $\varphi \in \widehat \W^{1,2}_{N_Q(x,R)}(Q (x,R))$.
Define $v \colon Q \to \IC$ by 
\[
v(y) 
= \left\{ \begin{array}{ll}
   \tilde v(y), & \mbox{if } y \in Q(x,R) ,  \\[5pt]
   0, & \mbox{if } y \in Q \setminus Q (x,R) .
          \end{array} \right.
\]
Then $v \in \widehat \W^{1,2}_{N_Q}(Q) \subset \W^{1,2}_{D_Q}(Q)$ 
by~\cite[Lem.~6.4]{ERe2}.
Set $w = u - \tilde v$.
Then $w \in \W^{1,2}_{D_Q}(Q)$.
Moreover, 
\[
\int_{Q (x,R)}
    \langle \mu^\Phi \nabla w, \nabla \varphi \rangle \; \d x
= 0
\]
for all $\varphi \in \widehat \W^{1,2}_{N_Q(x,R)}(Q (x,R))$ by~\eqref{epdir922;4}.
The De Giorgi inequalities applied to the function $w$
imply
\begin{eqnarray*}
\int_{Q (x,r)} |\nabla u|^2 \; \d x
& \leq & 2 \int_{Q(x,r)} |\nabla w|^2  \; \d x
       + 2 \int_{Q(x,r)} |\nabla \tilde v|^2 \; \d x \nonumber \\
& \leq & 2 c_{DG} \Big( \frac{r}{R} \Big)^{d-2+2\kappa_0} \int_{Q(x,R)} |\nabla w|^2 \; \d x
   + 2 \int_{Q (x,r)} |\nabla \tilde v|^2 \; \d x \nonumber  \\
& \leq & 4 c_{DG} \Big( \frac{r}{R} \Big)^{d-2+2\kappa_0} \int_{Q (x,R)} |\nabla u|^2 \; \d x
   + (2 + 4 c_{DG}) \int_{Q (x,R)} |\nabla \tilde v|^2   \; \d x
 .  \hspace{8mm}
\end{eqnarray*}
Choose $\varphi = \tilde v$ in (\ref{epdir922;4}).
Then 
\[ \int_{Q (x,R)} \langle \mu^\Phi \nabla \tilde v, \nabla \tilde v \rangle \; \d x
   = \int_{Q (x,R)} \langle \mu^\Phi \nabla u, \nabla \tilde v \rangle \; \d x
   = \int_{Q (x,R)} \langle (\mu^\Phi - (\mu + \nu)^\Phi) \nabla u , \nabla \tilde v \rangle  \; \d x. 
\]
Hence the estimate $c_\bullet - \tilde \varepsilon \geq \frac{c_\bullet}{2}$,
ellipticity on $Q$ (see~\cite{ERe2} Prop.~4.3(b))
and the Cauchy--Schwarz inequality give
\[
(d! M^{d+2})^{-1} \frac{c_\bullet}{2} \int_{Q (x,R)} |\nabla \tilde v|^2 \; \d x
\leq d! M^{d+2} \tilde \varepsilon \bigg(\int_{Q (x,R)} |\nabla \tilde u|^2 \; \d x \bigg)^{\frac{1}{2}}
    \bigg( \int_{Q (x,R)} |\nabla \tilde v|^2 \; \d x \bigg)^{\frac{1}{2}}
 .
\]
Therefore
\[
\int_{Q(x,R)} |\nabla \tilde v|^2  \; \d x
\leq (d! M^{d+2})^4 \frac{4 \tilde \varepsilon^2}{c_\bullet^2} 
    \int_{Q (x,R)} |\nabla \tilde u|^2 \; \d x
\]
and 
\begin{eqnarray*}
\int_{Q(x,r)} |\nabla u|^2 \; \d x
& \leq & 4 c_{DG} \Big( \frac{r}{R} \Big)^{d-2+2\kappa_0} \int_{Q(x,R)} |\nabla u|^2 \; \d x
     \\*
& & \hspace*{30mm} {}
   + (2 + 4 c_{DG}) (d! M^{d+2})^4 \frac{4 \tilde \varepsilon^2}{c_\bullet^2} \int_{Q (x,R)} |\nabla \tilde u|^2  \; \d x \\ 
& \leq & 4 c_{DG} \Big( \Big( \frac{r}{R} \Big)^{d-2+2\kappa_0} + \varepsilon \Big) \int_{Q (x,R)} |\nabla u|^2 \; \d x
 .  
\end{eqnarray*}
Hence by Lemma~\ref{ldir921} one concludes that 
\[
\int_{Q(x,r)} |\nabla u|^2 \; \d x
\leq \tilde c \Big( \frac{r}{R} \Big)^{d-2+\kappa_0} \int_{Q(x,R)} |\nabla u|^2 \; \d x
\]
for all $0 < r \leq R \leq R_0$.
The proposition follows by choosing $R = R_0$.
\end{proof}

The next proposition gives perturbed De Giorgi estimates near the  
Neumann part of the boundary, but away from the Dirichlet part of the boundary.

\begin{proposition} \label{pdir923}
Let $\Omega \subset \IR^d$ be open, $D \subset \partial \Omega$ be closed and $N = \partial \Omega \setminus D$ subject to Assumption~N. Let $x_0 \in N$
and let $\Phi_{x_0}$ be the bi-Lipschitz homeomorphism of Assumption~N whose corresponding set $U_{x_0}$ satisfies $U_{x_0} \cap \partial \Omega \subset N$.
Moreover, let 
$\mu \in \L^{\infty} (\Omega ; \IR^{d \times d})$ be subject to Assumption~\ref{Ass: Ellipticity}.

Then there exists an $\varepsilon > 0$ such that for all (complex valued)
$\nu \in \L^{\infty} (\Omega ; \IC^{d \times d})$ satisfying
$\|\nu\|_{\L^{\infty} (\Omega ; \Lop(\IC^d))} < \varepsilon$
the operator $A^{\mu + \nu}$ is m-sectorial.
Moreover, if $A_{\Phi}^{\mu + \nu}$ denotes the operator in $(-1,1)^{d-1} \times (0,1)$ obtained
from $A^{\mu + \nu}$ under the transformation $\Phi$, there are $\kappa_0 \in (0,1)$ and $c_{DG} > 0$ such that
$A_{\Phi}^{\mu + \nu}$ satisfies 
$(\kappa_0,c_{DG})$-De Giorgi estimates on $(-\frac{1}{2},\frac{1}{2})^{d-1} \times (0,\frac{1}{2})$ 
for functions vanishing on $\emptyset$
and Neumann boundary conditions on $(-1,1)^{d-1} \times \{ 0 \} $.
The constants $\varepsilon$, $\kappa_0$ and $c_{DG}$ can be chosen to depend
only on $d$, $M$, $c_\bullet$ and $c^\bullet$.
\end{proposition}
\begin{proof}
The proof is similar to the proof of Proposition~\ref{pdir922}, 
using~\cite[Lem.~5.1]{ERe2} instead of~\cite[Prop.~5.3]{ERe2}.
\end{proof}

Finally we need a perturbed version of~\cite[Prop.~2.1]{ERe2}.

\begin{proposition} \label{pdir924}
Let $\Omega\subset \IR^d$ be open, $D \subset \partial \Omega$ closed,
and set $N = \partial \Omega \setminus D$.
Let $\Upsilon \subset \overline \Omega$ be a subset.
Suppose there exist $\alpha, \zeta > 0$ such that 
$\dist (N,\Upsilon) \geq \zeta$ and 
$|B(x,r) \setminus \Omega| \geq \alpha r^d$
for all $r \in (0,1]$ and $x \in \partial \Omega$ with $\dist (x,\Upsilon) < \zeta$.

Then there exists an $\varepsilon > 0$ such that for all (complex valued)
$\nu \in \L^{\infty} (\Omega ; \IC^{d \times d})$ satisfying
$\|\nu\|_{\L^{\infty} (\Omega ; \Lop(\IC^d))} < \varepsilon$
the operator $A^{\mu + \nu}$ is m-sectorial.
Moreover there are $\kappa_0 \in (0,1)$ and $c_{DG} > 0$ such that the operator
$A^{\mu + \nu}$ satisfies 
$(\kappa_0,c_{DG})$-De Giorgi estimates on $\Upsilon$ 
for functions vanishing on $D$
and Neumann boundary conditions on $\emptyset$.
The constants $\varepsilon$, $\kappa_0$ and $c_{DG}$ can be chosen to depend
only on $d$, $\zeta$, $\alpha$, $c_\bullet$ and $c^\bullet$.
\end{proposition}
\begin{proof}
The proof is similar to the proof of Proposition~\ref{pdir922}, 
but restrict to the case $R_0 \leq \zeta$.
Of course this time one uses 
\cite[Prop.~2.1]{ERe2} instead of~\cite[Prop.~5.3]{ERe2}.
This gives the De Giorgi estimates for $0 < r \leq R \leq \zeta$.
The case $R \in (\zeta,1]$ follows easily from the case $R = \zeta$.
\end{proof}

\begin{proof}[Proof of Theorem~\ref{tnumrange307}.]
The proof is similar to the proof of~\cite[Thm.~7.5]{ERe2}, with the 
obvious changes to use Propositions~\ref{pdir922}, \ref{pdir923} 
and~\ref{pdir924}.
One deduces the Gaussian kernel bounds of Theorem~\ref{tnumrange307}.
These imply that the semigroup on $\L^2(\Omega)$ extends consistently to 
$\L^p(\Omega)$ for all $p \in [1,\infty)$.
\end{proof}

Finally, we present the proof of Corollary~\ref{Cor: Perturbation translated to p-ellipticity}.

\begin{proof}[Proof of Corollary~\ref{Cor: Perturbation translated to p-ellipticity}]
Let $\eps > 0$ be the $\eps$ from Theorem~\ref{tnumrange307} that belongs to
\textit{real valued} 
matrices with ellipticity constants $c_{\bullet} / 2$ and $2 c^{\bullet}$. 
Choose $p_c > 2$ such that
\begin{align*}
 \frac{(p_c - 1)^{1 / 2}}{p_c - 2} c^{\bullet} 
< \min \bigg\{ \eps , c^{\bullet} , \frac{c_{\bullet}}{2} \bigg\}.
\end{align*}
Now let $\mu \in \L^{\infty} (\Omega ; \IC^{d \times d})$ be subject to
Assumption~\ref{Ass: Ellipticity} and suppose that $p_0(\mu) > p_c$.
Then Lemma~\ref{Lem: Control of imaginary part} gives
\begin{align*}
 \| \Im (\mu) \|_{\L^{\infty} (\Omega ; \Lop (\IC^d))} 
\leq \frac{(p_0(\mu) - 1)^{1 / 2}}{p_0(\mu) - 2} c^{\bullet}
\leq \frac{(p_c - 1)^{1 / 2}}{p_c - 2} c^{\bullet} < \varepsilon.
\end{align*}
Therefore Theorem~\ref{tnumrange307} is applicable with $\mu$ replaced by $\Re
(\mu)$ and 
$\nu = \ii \Im (\mu)$, and the corollary follows.
\end{proof}

\section{Regularity for the induced operators on the $\W^{-1,q}_D$ scale}
\label{Sec: Regularity for the induced operators on the W^{-1,q}_D scale}

In Corollary~\ref{Cor: Maximal regularity on Lp} we 
investigated regularity properties for the divergence operators in
the $\L^p$ scale. In view of parabolic and general evolution equations this is
the most commonly used one, but in the treatment of real world problems there
are at least two effects that make this choice inadequate. These are, on the one
hand, inhomogeneous Neumann boundary conditions and, on the other hand,
reaction terms which live on lower dimensional manifolds. For instance, the
latter is of eminent importance when treating the semiconductor equations where
it is quite common to consider generation/recombination mechanisms that are
situated on surfaces. See~\cite[Section~4.2]{Selberherr} and
\cite[Section~3]{disser} for a detailed discussion from the mathematical
viewpoint. If the above-mentioned phenomena occur, the adequate spaces for the treatment of
semilinear and quasilinear parabolic equations are often spaces from the
$\W^{-1,q}_D$ scale or duals of Bessel potential spaces,
see the detailed discussion in~\cite[Section~6]{haller}.

For this purpose we will deduce parabolic regularity results on the
$\W^{-1,q}_D$ scale from the results received above. Following the philosophy of
\cite{haller} (see also~\cite[Section~11]{auscher} for more advanced ideas) we will
`transport' properties for the divergence operators $-\nabla \cdot \mu \nabla$,
formerly obtained on the $\L^p$ scale, into the $\W^{-1,q}_D$ scale by the
knowledge of the square root isomorphism
\begin{equation} \label{e-eg1}
 \bigl (-\nabla \cdot \mu \nabla +1 \bigr )^{-\frac {1}{2}} \colon \W^{-1,q}_D (\Omega) \to
	\L^q (\Omega).
\end{equation}
The isomorphism property of~\eqref{e-eg1} is deduced in~\cite{auscher} for
real coefficient functions $\mu$, while in his pioneering paper
\cite{egert_kato_Lp} Egert succeeded to prove the isomorphy~\eqref{e-eg1} for
complex coefficient functions (and even for systems) as long as the
corresponding semigroup is well-behaved on $\L^p (\Omega)$.

In order to go into details, we quote and apply the results from~\cite{egert_kato_Lp} that are
relevant for our purposes.

\begin{definition} \label{d-eg}
Let $\Omega \subset \IR^d$ be open and let $\mu \in \L^{\infty} (\Omega ; \IC^{d \times d})$ be 
subject to Assumption~\ref{Ass: Ellipticity}.
Let $J(\mu)$ be the interval of all numbers $p \in (1,\infty)$ for which
 the semigroup generated by $- A_2$ on $\L^2 (\Omega)$ extrapolates consistently 
to a bounded $C_0$-semigroup $S^{(p)}$
 on the space $\L^p(\Omega)$, that is 
$\sup_{t \in (0,\infty)} \|S^{(p)}_t\|_{\mathcal L(\L^p (\Omega))} < \infty$.
\end{definition}

The first result from \cite{egert_kato_Lp} is a bounded holomorphic 
functional calculus, see~\cite[Thm.~1.3]{egert_kato_Lp}.

\begin{theorem} \label{tnumrange702}
Let $\Omega \subset \IR^d$ be open and bounded, where $d \geq 3$.
Let $D \subset \partial \Omega$ be closed and set $N = \partial \Omega \setminus D$.
Adopt Assumption~N and let $\mu \in \L^{\infty} (\Omega ; \IC^{d \times d})$ be 
subject to Assumption~\ref{Ass: Ellipticity}. 
Let $\omega$ denote the half angle of sectoriality for the operator $A_2$.
Further, let $p_1 \in J(\mu)$ and $p \in (p_1,2) \cup (2,p_1)$.
Let $\zeta \in (\omega, \pi)$.
Then there exists a $c > 0$ such that 
\[
   \|f(A_2) u \|_{\L^p (\Omega)} \le c \|f\|_{\L^\infty(\Sec_\zeta)}
	\|u\|_{\L^p (\Omega)}, 
\quad (u \in \L^2 (\Omega) \cap \L^p (\Omega))
\]
for every bounded, holomorphic function $f$ on $\Sec_\zeta $.
The constant $c$ can be chosen to depend only on 
$d$, $M$, $c_\bullet$, $c^\bullet$, $\zeta$, $p$ and 
$\sup_{t \in (0,\infty)} \|S^{(p_1)}_t\|_{\mathcal L(\L^{p_1} (\Omega))}$.
\end{theorem}

The second is a topological isomorphism.

\begin{theorem} \label{t-eg1}
Let $\Omega \subset \IR^d$ be open and bounded, where $d \geq 3$.
Let $D \subset \partial \Omega$ be closed and set $N = \partial \Omega \setminus D$.
Suppose that $\Omega$ is a $d$-set and $D$ is a $(d-1)$-set.
Adopt Assumption~N and let $\mu \in \L^{\infty} (\Omega ; \IC^{d \times d})$ be 
subject to Assumption~\ref{Ass: Ellipticity}. 
Then for all
 $p \in (\frac{p_0 (\mu) d}{d(p_0 (\mu) - 1) + 2} ,2]$ the operator 
 \begin{equation} \label{e-wury}
   \bigl (A_p+1 \bigr )^{-\frac {1}{2}} \colon \L^p(\Omega) \to \W^{1,p}_D (\Omega)
 \end{equation}
 is a topological isomorphism.
\end{theorem}
\begin{proof}
This follows from Corollary~\ref{cnumrange302} together with \cite[Thm.~1.2(i)]{egert_kato_Lp}.
\end{proof}

In all what follows, for a given coefficient function $\mu$ and 
$p \in (\frac{p_0 (\mu) d}{d(p_0 (\mu) - 1) + 2} ,\frac{p_0 (\mu) d}{d - 2})$  
we denote by $A_p^{\mu}$ the operator corresponding to $\mu$.
Then $A_p^{(\mu^*)}$ denotes the operator which corresponds to the conjugate coefficient function 
$\mu ^*$. 
It is standard that $(A_2^{(\mu^*)})^* = A_2^{\mu}$.

Noting that $p_0(\mu)$ and $p_0(\mu^*)$ coincide (see Section~\ref{Sec: The elliptic operator}), 
one obtains the following corollary.

\begin{corollary} \label{c-wurz}
Adopt the notation and assumptions of Theorem~\ref{t-eg1} above.
Then for every $p \in \bigl( \frac{p_0 (\mu) d}{d(p_0 (\mu) - 1) + 2} , 2 \bigr]$ the 
adjoint map $\bigl (( A_p^{(\mu^*)} + 1)^{-\frac 12}\bigr )'$
is a topological isomorphism from $\W^{-1,q}_D (\Omega)$ onto $\L^q (\Omega)$
which is consistent with $(A_2^\mu + 1)^{-\frac 12}$.
\end{corollary}

\begin{definition} \label{d-defwminq}
Adopt the notation and assumptions of Theorem~\ref{t-eg1}.
 For all $q \in [2, \frac{p_0 (\mu) d}{d - 2})$ we denote 
by $E_q^\mu \colon \W^{-1,q}_D (\Omega) \to \L^q (\Omega)$ the adjoint map of
\[ \bigl( A_p^{(\mu^*)}+1 \bigr )^{-\frac {1}{2}} \colon \L^p(\Omega) \to \W^{1,p}_D (\Omega).
\]
\end{definition}
We emphasise that $E_q^\mu : W^{-1,q}_D(\Omega) \to L^q(\Omega)$ is a topological isomorphism.
Let $q \in [2, \frac{p_0 (\mu) d}{d - 2})$.
Then $-A_q^\mu$ is the generator of a $C_0$-semigroup $S^{(q)}$ on $\L^q(\Omega)$
by Corollary~\ref{cnumrange302}.
We use $E_q^\mu$ to transfer $S^{(q)}$ to a $C_0$-semigroup $T^{(q)}$ on 
$\W^{-1,q}_D (\Omega)$ defined by 
\[
T^{(q)}_t = (E_q^\mu)^{-1} \, S^{(q)}_t \, E_q^\mu
.  \]
We denote the generator of $T^{(q)}$ by $- B_q^\mu$.
Next, define $\mathcal{A} \colon \W^{1,2}_D (\Omega) \to \W^{-1,2}_D (\Omega)$
by $(\mathcal{A} u)(v) = \ta[u,v]$ for all $u,v \in \W^{1,2}_D (\Omega)$,
where $\mathfrak{t}$ was defined in~\eqref{Eq: Sesquilinear form}.
We consider $\mathcal{A}$ as a densily defined operator in $\W^{-1,2}_D (\Omega)$.

\begin{lemma} \label{lnumrange706}
Adopt the notation and assumptions of Theorem~\ref{t-eg1}.
Let $q \in [2, \frac{p_0 (\mu) d}{d - 2})$.
Then $B_q^\mu$ is the part of $\mathcal{A}$ in $\W^{-1,q}_D (\Omega)$, that is 
$\dom(B_q^\mu) = \{ u \in \W^{-1,q}_D (\Omega) : \mathcal{A} u \in \W^{-1,q}_D (\Omega) \} $
and $B_q^\mu u = \mathcal{A} u$ for all $u \in \dom(B_q^\mu)$.
\end{lemma}
\begin{proof}
This follows as in \cite[Lem.~6.9(c)]{DisserterElstRehberg2}.
\end{proof}

\begin{theorem} \label{t-imag}
Let $\Omega \subset \IR^d$ be open and bounded, where $d \geq 3$.
Let $D \subset \partial \Omega$ be closed and set $N = \partial \Omega \setminus D$.
Suppose that $\Omega$ is a $d$-set and $D$ is a $(d-1)$-set.
Adopt Assumption~N and let $\mu \in \L^{\infty} (\Omega ; \IC^{d \times d})$ be 
subject to Assumption~\ref{Ass: Ellipticity}. 
Let $\omega$ denote the half angle of sectoriality for the operator $A_2$.
Further let $q \in [2, \frac{p_0 (\mu) d}{d - 2})$.
Then one has the following.
 \begin{tabel}
\item \label{holo} For all $\zeta \in (\omega,
	\pi/2)$ the operator $B^\mu_q +1$
	admits a bounded holomorphic calculus on the sector $\Sec_\zeta$.
	In particular, $B^\mu_q +1$  admits bounded imaginary powers on
	$\W^{-1,q}_D (\Omega)$ and
	\[ \sup_{s \in [-1,1]} \|(B^\mu_q +1 )^{\ii s}\|_{\mathcal L(\W^{-1,q}_D (\Omega))} < \infty.
	\]
\item \label{maxpareg} The operator $B^\mu_q$ has maximal 
	regularity on $\W^{-1,q}_D (\Omega)$.
 \end{tabel}
\end{theorem}
\begin{proof}
`\ref{holo}'.
Since $[2, \frac{p_0 (\mu) d}{d - 2}) \subset J(\mu)$ by Corollary~\ref{cnumrange302},
it follows that the operator $A_q^\mu + 1$ has a bounded holomorphic calculus on 
$\L^q(\Omega)$ by Theorem~\ref{tnumrange702}.
Since $E_q^\mu$ is a topological isomorphism the bounded holomorphic calculus
transfers to $\W^{-1,q}_D (\Omega)$,
	with the same angle as for $A^\mu_q$, is implied by~\cite[Prop.~2.11]{Denk}. 
The boundedness of the purely imaginary powers
	follows from this, see~\cite[p.~25]{Denk}.

`\ref{maxpareg}'.
The operator $B^\mu_q + 1$ has maximal parabolic regularity by Statement~\ref{holo} and 
\cite{dore}.
Therefore also the operator $B^\mu_q$ has maximal parabolic regularity.
\end{proof}

\begin{rem} \label{r-bessel}
\mbox{}
 \begin{itemize}
  \item  It is known that maximal parabolic regularity is preserved under (real
	and complex) interpolation, see~\cite[Lemma~5.3]{haller}. Using the
    interpolation results from~\cite{Eb}, this shows that the
minus generator of the consistent $C_0$-semigroup on the (dual of) Bessel potential spaces
	with non-integer differentiability index between $-1$ and $0$ also satisfies maximal parabolic
	regularity. This considerably generalizes the results 
in~\cite[Thm.~5.16 and Section~6]{haller} in view of complex coefficients
	and much more general admissible geometries for $\Omega$ and $D$.
  \item If one is only interested in maximal parabolic regularity, it is also
	possible to transport this property directly from $\L^q (\Omega)$ to $\W^{-1,q}_D (\Omega)$
	by the square root isomorphism~\eqref{e-eg1}, see~\cite[Lem.~5.12]{HaR3}.
 \end{itemize}
\end{rem}

\end{document}